\documentclass[12pt,a4paper]{amsart}
\usepackage[abbrev]{amsrefs}
\usepackage{amscd}
\usepackage{color}
\usepackage{amssymb}
\usepackage{enumerate}

\theoremstyle{plain}
  \newtheorem{thm}{Theorem}[section]
  \newtheorem{lem}[thm]{Lemma}
  
  \newtheorem{cor}[thm]{Corollary}
  \newtheorem{prop}[thm]{Proposition}
  \newtheorem{clm}[thm]{Claim}

\theoremstyle{definition}
  \newtheorem{defn}[thm]{Definition}

\theoremstyle{remark}
  \newtheorem{rem}[thm]{Remark}

\numberwithin{equation}{section}

\makeatletter
\newcommand\RedeclareMathOperator{%
  \@ifstar{\def\rmo@s{m}\rmo@redeclare}{\def\rmo@s{o}\rmo@redeclare}%
}
\newcommand\rmo@redeclare[2]{%
  \begingroup \escapechar\m@ne\xdef\@gtempa{{\string#1}}\endgroup
  \expandafter\@ifundefined\@gtempa
     {\@latex@error{\noexpand#1undefined}\@ehc}%
     \relax
  \expandafter\rmo@declmathop\rmo@s{#1}{#2}}
\newcommand\rmo@declmathop[3]{%
  \DeclareRobustCommand{#2}{\qopname\newmcodes@#1{#3}}%
}
\@onlypreamble\RedeclareMathOperator
\makeatother

\newcommand{\field}[1]{\mathbb{#1}}
\newcommand{\C}{\field{C}}
\newcommand{\R}{\field{R}}

\newcommand{\N}{\field{N}}
\renewcommand{\H}{\field{H}}
\DeclareMathOperator{\ObsDiam}{ObsDiam}
\DeclareMathOperator{\diam}{diam}

\DeclareMathOperator{\supp}{supp}
\DeclareMathOperator{\dP}{\mathit{d}_\mathrm{P}}
\DeclareMathOperator{\dconc}{\mathit{d}_\mathrm{conc}}

\DeclareMathOperator{\dH}{\mathit{d}_\mathrm{H}}
\DeclareMathOperator{\dTV}{\mathit{d}_\mathrm{TV}}
\DeclareMathOperator{\vol}{vol}
\DeclareMathOperator{\diag}{diag}
\DeclareMathOperator{\Ree}{R}
\DeclareMathOperator{\I}{I}
\DeclareMathOperator{\J}{J}
\DeclareMathOperator{\K}{K}

\RedeclareMathOperator{\i}{i}
\RedeclareMathOperator{\j}{j}
\RedeclareMathOperator{\k}{k}

\RedeclareMathOperator{\Re}{Re}

\newcommand{\dKF}[1][]{{d_{\rm KF}^{#1}}}
\newcommand{\PV}{\mathrm{P}V}
\newcommand{\cX}{\mathcal{X}}
\newcommand{\cL}{\mathcal{L}}
\newcommand{\Lip}{\mathcal{L}\mathit{ip}}
\newcommand{\cP}{\mathcal{P}}

\DeclareMathOperator{\tr}{tr}
\DeclareMathOperator{\T}{{}^{T}\!\!}

\newcommand{\lr}[2]{\left\langle{#1},{#2}\right\rangle}
\begin{document}
\title{High-dimensional metric-measure limit of Stiefel and Grassmann manifolds}

\author{Takashi Shioya}
\address{Mathematical Institute, Tohoku University, Sendai 980--8578,
  JAPAN}
\email{shioya@math.tohoku.ac.jp}

\author{Asuka Takatsu}
\address{Department of Mathematics and Information Sciences, Tokyo Metropolitan University, Tokyo 192--0397,
JAPAN}
\email{asuka@tmu.ac.jp}

\date{\today}



\thanks{This work was supported by JSPS KAKENHI Grant Numbers
26400060, 15K17536, 15H05739.}

\maketitle

\begin{abstract}
We study the high-dimensional limit of (projective) Stiefel and Grassmann manifolds as metric measure spaces in Gromov's topology.
The limits are either the infinite-dimensional Gaussian space
or its quotient by some mm-isomorphic group actions,
which are drastically different from the manifolds.
As a corollary, we obtain some asymptotic estimates
of the observable diameter of (projective) Stiefel and Grassmann manifolds.
\end{abstract}


\section{Introduction}

Gromov developed the metric geometric theory of metric measure spaces,
say \emph{mm-spaces} (see \cite{Gro:green,Sy:mmg}).
There he defined a concept of convergence,
which we prefer to call \emph{weak convergence},
of mm-spaces
by the convergence of the sets of $1$-Lipschitz continuous functions on the spaces.
This is a weaker version of
measured Gromov-Hausdorff convergence.
In fact, a measured Gromov-Hausdorff convergent
sequence converges weakly to the same limit, however
the converse does not necessarily hold.
The idea of the definition of weak convergence came from
the concentration of measure phenomenon due to L\'evy and Milman
(see \cite{GroMil,Levy,Mil:heritage,Le}).
In fact, since any function
on a one-point mm-space is constant,
a sequence $\{X_n\}_{n=1}^\infty$ of mm-spaces converges weakly to
a one-point mm-space
if and only if
any $1$-Lipschitz function on $X_n$ is almost constant
for all sufficiently large $n$, which is just the concentration of measure
phenomenon.
L\'evy's celebrated lemma \cite{Levy} is rephrased as that the sequence of unit spheres in
Euclidean spaces
converges weakly to a one-point mm-space as dimension diverges to infinity.

It is a natural problem to study the weak limit of
a non-Gromov-Hausdorff precompact sequence of specific manifolds,
such as homogeneous manifolds, as dimension diverges to infinity.
In our main theorems,
we observe that the high-dimensional weak limits of
(projective) Stiefel and Grassmann manifolds
are drastically different from
the manifolds in the sequence.
This kind of phenomenon is never seen in the Gromov-Hausdorff
convergence.

For $n=1,2,\dots,\infty$, we call the mm-space
$\Gamma^n := (\R^n,\|\cdot\|,\gamma^n)$
the \emph{$n$-dimensional} (\emph{standard}) \emph{Gaussian space},
where $\|\cdot\|$ is the $l^2$ (or Euclidean) norm
and $\gamma^n$ the $n$-dimensional standard Gaussian measure
on $\R^n$.  If $n = \infty$, then $\|\cdot\|$ takes values in $[\,0,+\infty\,]$.
Let $F$ be one of $\R$, $\C$, and $\H$, where $\H$ is the algebra of quaternions,
and let $M_{N,n}^F$ denote the set of $N \times n$ matrices over $F$.
Recall that the \emph{$(N,n)$-Stiefel manifold over $F$}, say $V^F_{N,n}$,
is defined to be
the submanifold of $M_{N,n}^F$ consisting of matrices with
orthonormal column vectors.
We equip Stiefel manifolds with the mm-structure induced from
the Frobenius norm and the Haar probability measure.
We set $N^F := N \cdot \dim_\R F$ for a number $N$.
Let $\{n_N\}_{N=1}^\infty$ be a sequence of positive integers
such that $n_N \le N$ for all $N$.
We consider the following condition ($*$) for the sequence $\{n_N\}$.
\[
\text{There exists a number $c \in (0,1)$ such that}
\ \ %
\sup_N \left( 2\log n_N - \frac{c}{4} \sqrt{\frac{N}{n_N^3}} \right) < \infty.
\tag{$*$}
\]
Note that we have ($*$) for $n_N = O(N^p)$, $0 \le p < 1/3$ (see Remark \ref{rem:condi}\eqref{c3}).
One of our main theorems is stated as follows.

\begin{thm} \label{thm:Stiefel}
If $\{n_N\}$ satisfies {\rm($*$)},
then the $(N,n_N)$-Stiefel manifold $V_{N,n_N}^F$ with distance multiplied by $\sqrt{N^F-1}$ 
converges weakly to the infinite-dimensional Gaussian space $\Gamma^\infty$ as $N \to \infty$.
\end{thm}

Note that the infinite-dimensional Gaussian space
is not an mm-space in the ordinary sense
and is defined as an element of a natural compactification
of the space of isomorphism classes of mm-spaces.
In the case where $n_N = 1$, the manifold $V_{N,n_N}^F = V_{N,1}^F$ is
a sphere in a Euclidean space, for which the weak limit
was obtained in \cite{Sy:mmg,Sy:mmlim}.

To specify the high-dimensional weak limit of Grassmann and projective Stiefel manifolds,
we need some notations.
Denote by $U^F(n)$ the $F$-unitary group of size $n$,
i.e., $U^F(n)$ is the orthogonal group if $F = \R$,
the complex unitary group if $F = \C$, and the quaternionic unitary group
if $F = \H$.
$U^F(n)$ acts on $M_{N,n}^F$ by right multiplication.
We define the infinite-dimensional Gaussian measures
on $M_{\infty,n}^F$ and $F^\infty$
by identifying $M_{\infty,n}^F$ and $F^\infty$ with $\R^\infty$
under natural isomorphisms.
The \emph{$U^F(1)$-Hopf action on $F^\infty$}
is defined as the $U^F(1)$ left multiplication.
Note that the \emph{$(N,n)$-Grassmann manifold over $F$}, say $G^F_{N,n}$,
is obtained
as the quotient of the $(N,n)$-Stiefel manifold over $F$ by the $U^F(n)$-action,
and that the \emph{projective $(N,n)$-Stiefel manifold over $F$}, say $\PV^F_{N,n}$,
is obtained
as the quotient of the $(N,n)$-Stiefel manifold over $F$
by the $U^F(1)$-Hopf action.
Note also that the $U^F(n)$-action and the $U^F(1)$-Hopf action
on the $(N,n)$-Stiefel manifold are both mm-isomorphic.
We equip Grassmann and projective Stiefel manifolds
with the quotient metric of the Frobenius
norm and the Haar probability measure (see Definition \ref{defn:quotient}).
Applying (the proof of) Theorem \ref{thm:Stiefel} we prove the following theorem.

\begin{thm} \label{thm:Gr-pS}
\begin{enumerate}
\item For any fixed positive integer $n$, as $N \to \infty$,
the $(N,n)$-Grassmann manifold $G_{N,n}^F$ over $F$
with distance multiplied by $\sqrt{N^F-1}$
converges weakly
to the quotient of
$(M_{\infty,n}^F,\|\cdot\|,\gamma^\infty)$
by the $F$-unitary group $U^F(n)$ of size $n$,
where $\|\cdot\|$ is the Frobenius norm.
\item If $\{n_N\}$ satisfies {\rm($*$)}, then
the projective $(N,n_N)$-Stiefel manifold $\PV_{N,n_N}^F$ over $F$
with distance multiplied by $\sqrt{N^F-1}$
converges weakly
to the quotient
of the infinite-dimensional Gaussian space $(F^\infty,\|\cdot\|,\gamma^\infty)$
by the $U^F(1)$-Hopf action.
\end{enumerate}
\end{thm}


The reason why we fix $n$ in Theorem \ref{thm:Gr-pS}(1)
is that we do not know the weak convergence of the quotient space of 
$(M_{\infty,n}^F,\|\cdot\|,\gamma^\infty)$ by $U^F(n)$ as $n\to\infty$.
If it converges weakly, then the $(N,n_N)$-Grassmann manifold
also converges weakly to the same limit, provided $\{n_N\}$ satisfies ($*$).

In the case where $n_N = 1$, the projective $(N,1)$-Stiefel manifold over $F$
is just the projective space over $F$, for which
the weak limit was obtained in \cite{Sy:mmlim}.

Theorems \ref{thm:Stiefel} and \ref{thm:Gr-pS}
are also true for any subsequence of $\{N\}$.
The proofs are the same.

The observable diameter of an mm-space is a quantity of how much
the measure of the mm-space concentrates
(see Definition \ref{defn:ObsDiam}).
As a corollary to the above theorems, we have
some asymptotic estimates for the observable diameter of
(projective) Stiefel and Grassmann manifolds.

\begin{cor} \label{cor:ObsDiam}
If $\{n_N\}$ satisfies {\rm($*$)}, then we have, for any $0 < \kappa < 1$,
\begin{align*}
\lim_{N\to\infty} \sqrt{N^F} \ObsDiam(V_{N,n_N}^F;-\kappa)
&= 2\,D^{-1}\left(1-\frac{\kappa}{2}\right), \tag{1}\\
\limsup_{N\to\infty} \sqrt{N^F} \ObsDiam(G_{N,n_N}^F;-\kappa)
&\le 2\,D^{-1}\left(1-\frac{\kappa}{2}\right), \tag{2}\\
\liminf_{N\to\infty} \sqrt{N^F} \ObsDiam(G_{N,n}^F;-\kappa)
&\ge e^{1/2}(1-\kappa) \qquad\text{for any fixed $n$}, \tag{3}\\
\limsup_{N\to\infty} \sqrt{N^F} \ObsDiam(\PV_{N,n_N}^F;-\kappa)
&\le 2\,D^{-1}\left(1-\frac{\kappa}{2}\right), \tag{4}\\
\liminf_{N\to\infty} \sqrt{N^F} \ObsDiam(\PV_{N,m_N}^F;-\kappa)
&\ge e^{1/2}(1-\kappa) \quad\text{for any $\{m_N\}$ with $m_N \le N$}, \tag{5}
\end{align*}
where
\[
D(r) := \gamma^1((\,-\infty,r\,]) = \int_{-\infty}^r \frac{1}{\sqrt{2\pi}} e^{-\frac{x^2}{2}}\;dx
\]
is the cumulative distribution function of $\gamma^1$.
\end{cor}

Some upper bounds of the observable diameter of Stiefel and
Grassmann manifolds were essentially obtained by Milman \cite{Mil:asymp,Mil:inf-dim}
and Milman-Schechtman \cite{MS} formerly.
Corollary \ref{cor:ObsDiam}(1) gives an asymptotically optimal estimate.
(2) and (4) are direct consequences of (1).
(3) is a nontrivial result.
As far as the authors know, any lower estimate of the observable diameter
of the Grassmann manifold was not known before.

The $(N,n)$-Stiefel manifold over $F$ is naturally embedded
into $M_{N,n}^F$.
We point out that just to compare the distance between
the Haar probability measure on the $(N,n)$-Stiefel manifold
and the Gaussian measure on $M_{N,n}^F$ is not enough to obtain
Theorem \ref{thm:Stiefel}.  In fact we have the following

\begin{thm}\label{prok}
The Prohorov distance between the Haar probability measure
on the $(N,n_N)$-Stiefel manifold over $F$ with distance multiplied by $\sqrt{N^F-1}$
and the Gaussian measure $\gamma^{N^F n_N}$ on $M_{N,n_N}^F$ is bounded away
from zero for all $N = 1,2,\dots$.
\end{thm}

Theorems \ref{thm:Stiefel} and \ref{prok} tell us that
the weak convergence of mm-spaces is different from the weak convergence of measures.

The idea of the proof of Theorem \ref{thm:Stiefel} is as follows.
Denote by $X^F_{N,n}$ the $(N,n)$-Stiefel manifold over $F$
with distance multiplied by $\sqrt{N^F - 1}$.
It suffices to prove that
\begin{align*}
\lim_{N\to\infty} X^F_{N,n_N} &\prec \Gamma^\infty, \tag{i}\\
\lim_{N\to\infty} X^F_{N,n_N} &\succ \Gamma^\infty, \tag{ii}
\end{align*}
where $\prec$ is the Lipschitz order relation (see Definition \ref{defn:dom}).
Note that the Lipschitz order relation naturally extends to the relation
on the compactification of the space of mm-spaces.

(ii) follows from an easy discussion.
$S^n(r)$ denotes an $n$-dimensional sphere of radius $r$ in a Euclidean space.
We have
\[
X^F_{N,n_N} \succ S^{N^F}(\sqrt{N^F -1})
\]
and the Maxwell-Boltzmann distribution law implies
\[
\lim_{N\to\infty} S^{N^F}(\sqrt{N^F -1}) \succ \Gamma^k
\]
for any $k$.
Combining these leads to (ii).

For the proof of (i), we find a suitable neighborhood of $X^F_{N,n_N}$ in $M_{N,n_N}^F$
which has most of the total measure of $\gamma^{N^F n_N}$,
so that the neighborhood approximates
$(M_{N,n_N}^F,\|\cdot\|,\gamma^{N^F n_N})$.
We estimate the Lipschitz constant of the nearest point projection from
the neighborhood to $X^F_{N,n_N}$
by using the polar decomposition of a matrix in the neighborhood,
where the smallest Lipschitz constant is eventually close to one.
We need delicate estimates of the measure of the neighborhood
and the Lipschitz constant to justify the proof of (i),
in which we find out that the condition ($*$)
guarantees to our estimates.
Note that we do not know the weak limit without ($*$).

Theorem \ref{thm:Gr-pS} is proved
by using Theorem \ref{thm:Stiefel}.
To prove it,
we need maps between $X^F_{N,n_N}$ and a finite-dimensional approximation of $\Gamma^\infty$ that are
equivariant with respect to the $U^F(n)$ and $U^F(1)$-Hopf actions.
We need the generalization of the Maxwell-Boltzmann distribution law
due to Watson \cite{Watson} to obtain such maps
for the case of Grassmann manifolds.

Our main theorems could be related with the infinite-dimensional analysis, such as
the theory of abstract Wiener spaces.  In fact, the infinite-dimensional Gaussian space
$\Gamma^\infty$
is an abstract Wiener space with $l^2$ as its Cameron-Martin space.
However, $\Gamma^\infty$ admits no separable Banach norm fitting to $\gamma^\infty$
and is not mm-isomorphic to any abstract Wiener space with separable Banach norm.
By this reason, many useful theorems in the theory of abstract Wiener spaces
cannot be applied to $\Gamma^\infty$.
We conjecture that the weak limit of compact homogeneous Riemannian manifolds
is not mm-isomorphic to any abstract Wiener space with separable Banach norm.

\section{Preliminaries}
\subsection{Metric measure geometry}

In this subsection, we give the definitions and the facts
stated in \cite{Gro:green}*{\S 3$\frac12$} and \cite{Sy:mmg}.
The reader is expected to be familiar with basic measure theory
and metric geometry (cf.~\cite{Kechris, Bil, Bog, BBI}).

\subsubsection{mm-Isomorphism and Lipschitz order}

\begin{defn}[mm-Space]
  Let $(X,d_X)$ be a complete separable metric space
  and $\mu_X$ a Borel probability measure on $X$.
  We call the triple $(X,d_X,\mu_X)$ an \emph{mm-space}.
  We sometimes say that $X$ is an mm-space, in which case
  the metric and the measure of $X$ are respectively indicated by
  $d_X$ and $\mu_X$.
\end{defn}

\begin{defn}[mm-Isomorphism]
  Two mm-spaces $X$ and $Y$ are said to be \emph{mm-isomorphic}
  to each other if there exists an isometry $f : \supp\mu_X \to \supp\mu_Y$
  such that $f_\#\mu_X = \mu_Y$,
  where $f_\#\mu_X$ is the push-forward of $\mu_X$ by $f$
  and $\supp\mu_X$ the support of $\mu_X$.
  Such an isometry $f$ is called an \emph{mm-isomorphism}.
  Denote by $\cX$ the set of mm-isomorphism classes of mm-spaces.
\end{defn}

Any mm-isomorphism between mm-spaces is automatically surjective,
even if we do not assume it.
Note that $X$ is mm-isomorphic to $(\supp\mu_X,d_X,\mu_X)$.

\emph{We assume that an mm-space $X$ satisfies
\[
X = \supp\mu_X
\]
unless otherwise stated.}

\begin{defn}[Lipschitz order] \label{defn:dom}
  Let $X$ and $Y$ be two mm-spaces.
  We say that $X$ (\emph{Lipschitz}) \emph{dominates} $Y$
  and write $Y \prec X$ if
  there exists a $1$-Lipschitz map $f : X \to Y$ satisfying
  $f_\#\mu_X = \mu_Y$.
  We call the relation $\prec$ on $\cX$ the \emph{Lipschitz order}.
\end{defn}

\begin{prop} \label{prop:Liporder}
  The Lipschitz order $\prec$ is a partial order relation on $\cX$, i.e.,
  we have the following {\rm(1)}, {\rm(2)}, and {\rm(3)}
  for any mm-spaces $X$, $Y$, and $Z$.
  \begin{enumerate}
  \item $X \prec X$.
  \item If $X \prec Y$ and $Y \prec X$, then $X$ and $Y$ are
    mm-isomorphic to each other.
  \item If $X \prec Y$ and $Y \prec Z$, then $X \prec Z$.
  \end{enumerate}
\end{prop}

\subsubsection{Observable diameter}

The observable diameter is one of the most fundamental invariants
of an mm-space.

\begin{defn}[Partial and observable diameter] \label{defn:ObsDiam}
  Let $X$ be an mm-space and let $\kappa > 0$.
  We define
  the \emph{partial diameter
  $\diam(X;1-\kappa) = \diam(\mu_X;1-\kappa)$ of $X$}
  to be the infimum of $\diam A$,
  where $A \subset X$ runs over all Borel subsets
  with $\mu_X(A) \ge 1-\kappa$ and $\diam A$ denotes the diameter of $A$.
  Denote by $\Lip_1(X)$ the set of $1$-Lipschitz continuous
  real-valued functions on $X$.
  We define
  the \emph{observable diameter of $X$} to be
  \[
  \ObsDiam(X;-\kappa) := \sup_{f \in \Lip_1(X)} \diam(f_\#\mu_X;1-\kappa).
  \]
\end{defn}


\begin{prop} \label{prop:ObsDiam-dom}
  If $X \prec Y$ for two mm-spaces $X$ and $Y$, then
  \[
  \ObsDiam(X;-\kappa) \le \ObsDiam(Y;-\kappa)
  \]
  for any $\kappa > 0$.
\end{prop}

\subsubsection{Distance between measures}

\begin{defn}[Total variation distance]
  The \emph{total variation distance $\dTV(\mu,\nu)$} of
  two Borel probability measures $\mu$ and $\nu$ on a topological space $X$
  is defined by
  \[
  \dTV(\mu,\nu) := \sup_A |\,\mu(A) - \nu(A)\,|,
  \]
  where $A$ runs over all Borel subsets of $X$.
\end{defn}

If $\mu$ and $\nu$ are both absolutely continuous with respect to
a Borel measure $\omega$ on $X$, then
\[
\dTV(\mu,\nu) = \frac{1}{2} \int_X \left| \frac{d\mu}{d\omega} - \frac{d\nu}{d\omega} \right| \; d\omega,
\]
where $\frac{d\mu}{d\omega}$ is the Radon-Nikodym derivative of $\mu$ with respect to $\omega$.

\begin{defn}[Prohorov distance]
  The \emph{Prohorov distance} $\dP(\mu,\nu)$ between two Borel probability
    measures $\mu$ and $\nu$ on a metric space $X$
  is defined to be the infimum of $\varepsilon > 0$ satisfying
  \[
  \mu(B_\varepsilon(A)) \ge \nu(A) - \varepsilon
  \]
  for any Borel subset $A \subset X$, where
  \[
  B_\varepsilon(A) := \{\; x \in X \mid d_X(x,A) < \varepsilon\;\}.
  \]
\end{defn}

The Prohorov metric is a metrization of weak convergence of
Borel probability measures on $X$ provided that $X$ is a separable
metric space.

\begin{prop} \label{dpdtv}
  For any two Borel probability
  measures $\mu$ and $\nu$ on a metric space $X$, we have
  $\dP(\mu,\nu) \le \dTV(\mu,\nu)$.
\end{prop}

\begin{defn}[Ky Fan distance]
  Let $(X,\mu)$ be a measure space and $Y$ a metric space.
  For two $\mu$-measurable maps $f,g : X \to Y$, we define 
  the \emph{Ky Fan distance $\dKF(f,g)$ between $f$ and $g$}
  to be the infimum of $\varepsilon \ge 0$ satisfying
  \[
  \mu(\{\;x \in X \mid d_Y(f(x),g(x)) > \varepsilon\;\}) \le \varepsilon.
  \]
\end{defn}

$\dKF$ is a pseudo-metric on the set of $\mu$-measurable maps from $X$ to $Y$.
It follows that $\dKF(f,g) = 0$ if and only if $f = g$ $\mu$-a.e.


\subsubsection{Box distance and observable distance}

\begin{defn}[Parameter]
  Let $I := [\,0,1\,)$ and let $X$ be an mm-space.
  A map $\varphi : I \to X$ is called a \emph{parameter of $X$}
  if $\varphi$ is a Borel measurable map such that
  $\varphi_\#\cL^1 = \mu_X$,
  where $\cL^1$ denotes the one-dimensional Lebesgue measure on $I$.
\end{defn}

It is known that any mm-space has a parameter.

\begin{defn}[Box distance]
  We define the \emph{box distance $\square(X,Y)$ between
    two mm-spaces $X$ and $Y$} to be
  the infimum of $\varepsilon \ge 0$
  satisfying that there exist parameters
  $\varphi : I \to X$, $\psi : I \to Y$, and
  a Borel subset $I_0 \subset I$ such that
  \[
    \cL^1(I_0) \ge 1-\varepsilon \quad\text{and}\quad
    |\,\varphi^*d_X(s,t)-\psi^*d_Y(s,t)\,| \le \varepsilon
  \]
  for any $s,t \in I_0$,
  where
  $\varphi^*d_X(s,t) := d_X(\varphi(s),\varphi(t))$ for $s,t \in I$.
\end{defn}

The box metric $\square$ is a complete separable metric on $\cX$.

\begin{prop} \label{prop:box-dP}
  Let $X$ be a complete separable metric space.
  For any two Borel probability measures $\mu$ and $\nu$ on $X$,
  we have
  \[
  \square((X,\mu),(X,\nu)) \le 2 \dP(\mu,\nu).
  \]
\end{prop}

\begin{defn}[Observable distance] \label{defn:dconc}
  For any parameter $\varphi$ of $X$, we set
  \[
  \varphi^*\Lip_1(X)
  := \{\;f\circ\varphi \mid f \in \Lip_1(X)\;\}.
  \]
  We define the \emph{observable distance $\dconc(X,Y)$ between
    two mm-spaces $X$ and $Y$} by
  \[
  \dconc(X,Y) := \inf_{\varphi,\psi} \dH(\varphi^*\Lip_1(d_X),\psi^*\Lip_1(d_Y)),
  \]
  where $\varphi : I \to X$ and $\psi : I \to Y$ run over all parameters
  of $X$ and $Y$, respectively,
  and where $\dH$
  is the Hausdorff metric with respect to the Ky Fan metric $\dKF$
  for the one-dimensional Lebesgue measure on $I$.
  $\dconc$ is a metric on $\cX$.
  We say that a sequence $\{X_n\}_{n=1}^\infty$ of mm-spaces
  \emph{concentrates} or \emph{converges weakly} to an mm-space $X$ if $X_n$ $\dconc$-converges to $X$
  as $n\to\infty$.
\end{defn}


\begin{prop} \label{prop:dconc-box}
  For any two mm-spaces $X$ and $Y$ we have
  $\dconc(X,Y) \le \square(X,Y)$.
\end{prop}

\subsubsection{Group action} \label{ssec:quotient}

Let $X$ be a metric space and $G$ a group acting on $X$ isometrically.
Let $\bar{X} = X/G$ be the quotient space of $X$ by the $G$-action.
Denote by $\bar{x}$ the class in $\bar{X}$ represented by a point $x \in X$.
We define a pseudo-metric $d_{\bar{X}}$ on the quotient space $\bar{X}$ by
\[
d_{\bar{X}}(\bar{x},\bar{y}) := \inf_{g,h \in G} d_X(g\cdot x,h \cdot y),
\qquad \bar{x},\bar{y} \in \bar{X}.
\]
$d_{\bar{X}}$ is a metric if every orbit of $G$ is closed in $X$.

Let $X$ and $Y$ be two metric spaces and $G$ a group
acting on $X$ and $Y$ isometrically.
For any $G$-equivariant map $f : X \to Y$
(i.e., $g\cdot f(x) = f(g\cdot x)$),
we have a unique map $\bar{f} : \bar{X} \to \bar{Y}$
with the property that $\bar{f}(\bar{x}) = \overline{f(x)}$
for any $x \in X$.  We call $\bar{f}$ the \emph{quotient map of $f$}.

\begin{lem}\label{lem:equivdomi}
  For any $L > 0$ and for any $G$-equivariant $L$-Lipschitz map $f : X \to Y$,
  the quotient map
  $\bar{f}$ is also $L$-Lipschitz.
\end{lem}

\begin{proof}
  For any $x, y \in X$, we see
\begin{align*}
  & d_{\bar{Y}}(\bar{f}(\bar{x}),\bar{f}(\bar{y})) = d_{\bar{Y}}(\overline{f(x)},\overline{f(y)})
  = \inf_{g,h \in G} d_Y(g\cdot f(x),h\cdot f(y)) \\
  &= \inf_{g,h \in G} d_Y(f(g\cdot x),f(h\cdot y)) \le L \inf_{g,h \in G} d_X(g\cdot x,h\cdot y)
  = L \, d_{\bar{X}}(\bar{x},\bar{y}).
\end{align*}
This completes the proof.
\end{proof}

For a Borel measure $\mu$ on $X$,
we denote by $\bar\mu$ the push-forward measure of $\mu$
by the natural projection $X \to \bar{X}$.

\begin{lem}[\cite{Sy:mmlim}*{Lemma 5.9}] \label{lem:dP-quotient}
  Let $X$ be a metric space and $G$ a group acting on $X$ isometrically.
  Then, for any two Borel probability measures $\mu$ and $\nu$ on $X$,
  we have
  $\dP(\bar{\mu},\bar{\nu}) \le \dP(\mu,\nu)$.
\end{lem}

\begin{defn}[Quotient mm-space] \label{defn:quotient}
Let $X$ be an mm-space and $G$ a group acting on $X$ isometrically
such that every orbit is closed in $X$.
We equip the quotient space $\bar{X}$ with $d_{\bar{X}}$ and $\mu_{\bar{X}} := \bar\mu_X$,
and call it the \emph{quotient mm-space of $X$ by the $G$-action}.
\end{defn}

\subsubsection{Pyramid}

\begin{defn}[Pyramid] \label{defn:pyramid}
  A subset $\cP \subset \cX$ is called a \emph{pyramid}
  if it satisfies the following (1), (2), and (3).
  \begin{enumerate}
  \item If $X \in \cP$ and if $Y \prec X$, then $Y \in \cP$.
  \item For any two mm-spaces $X, X' \in \cP$,
    there exists an mm-space $Y \in \cP$ such that
    $X \prec Y$ and $X' \prec Y$.
  \item $\cP$ is nonempty and $\square$-closed.
  \end{enumerate}
  We denote the set of pyramids by $\Pi$.

  For an mm-space $X$ we define
  \[
  \cP_X := \{\;X' \in \cX \mid X' \prec X\;\},
  \]
  which is a pyramid.
  We call $\cP_X$ the \emph{pyramid associated with $X$}.
\end{defn}

We observe that $X \prec Y$ if and only if $\cP_X \subset \cP_Y$.
It is trivial that $\cX$ is a pyramid.

\begin{defn}[Weak convergence] \label{defn:w-conv}
  Let $\cP_n, \cP \in \Pi$, $n=1,2,\dots$.
  We say that \emph{$\cP_n$ converges weakly to $\cP$}
  as $n\to\infty$
  if the following (1) and (2) are both satisfied.
  \begin{enumerate}
  \item For any mm-space $X \in \cP$, we have
    $\lim_{n\to\infty} \square(X,\cP_n) = 0$.
  \item For any mm-space $X \in \cX \setminus \cP$, we have
    $\liminf_{n\to\infty} \square(X,\cP_n) > 0$.
  \end{enumerate}
\end{defn}

For an mm-space $X$, a pyramid $\cP$, and $t > 0$, we define
\[
tX := (X,t\,d_X,\mu_X) \quad\text{and}\quad
t\cP := \{\; tX \mid X \in \cP \;\}.
\]
The following is obvious.

\begin{lem}\label{lem:obvious}
\begin{enumerate}
\item Let $\cP$ and $\cP_n$, $n=1,2,\dots$, be pyramids, and let $t, t_n$ be positive real numbers.
If $t_n \to t$ and 
$\cP_n$ converges weakly to $\cP$ as $n\to\infty$, then $t_n\cP_n$ converges weakly to $t\cP$ as $n\to\infty$.
\item If $\{X_n\}_{n=1}^\infty$ is a monotone increasing sequence of mm-spaces with respect to the Lipschitz order, 
then $\cP_{X_n}$ converges weakly to the $\Box$-closure of the union of $\cP_{X_n}$.
\end{enumerate}
\end{lem}

\begin{thm} \label{thm:emb}
There exists a metric $\rho$ on $\Pi$ compatible with weak convergence
and satisfying the following {\rm(1)}, {\rm(2)}, and {\rm(3)}.
\begin{enumerate}
\item The map
$\iota : \cX \ni X \mapsto \cP_X \in \Pi$
is a $1$-Lipschitz topological embedding map with respect to
$\dconc$ and $\rho$.
\item\label{cpt} $\Pi$ is $\rho$-compact.
\item $\iota(\cX)$ is $\rho$-dense in $\Pi$.
\end{enumerate}
In particular, $(\Pi,\rho)$ is a compactification of $(\cX,\dconc)$.
\end{thm}

Note that we identify $X$ with $\cP_X$ in \S 1.

Combining Propositions \ref{dpdtv}, \ref{prop:box-dP}, \ref{prop:dconc-box},
and Theorem \ref{thm:emb} yields the following

\begin{cor} \label{cor:rho-dTV}
  For any two Borel probability measures $\mu$ and $\nu$ on
  a complete separable metric space $X$, we have
  \begin{align*}
  \rho(\cP_{(X,\mu)},\cP_{(X,\nu)}) &\le \dconc((X,\mu),(X,\mu)) \le \square((X,\mu),(X,\nu)) \\
   &\le 2\dP(\mu,\nu) \le 2\dTV(\mu,\nu).
  \end{align*}
\end{cor}

\subsection{Decompositions of real, complex, and quaternion matrices}
Let $F$ be one of $\R, \C$ and $\H$, where $\H$ is 
the non-commutative algebra $\H$ of quaternions, 
which is defined as
\[
\H:=\{z:=z_0 + z_1 \i+ z_2 \j +z_3 \k \ |\ z_0, z_1,z_2, z_3 \in \R \}, \quad \i^2=\j^2=\k^2=\i\j\k=-1.
\]
Note that $\R$ and $\C$ are naturally embedded into $\H$.
For $z =z_0 + z_1 \i+ z_2 \j +z_3 \k \in \H$, 
we define 
\begin{gather*}
\Re(z):=z_0, \quad
z^\ast:=z_0-(z_1 \i+ z_2 \j +z_3 \k ).
\end{gather*}
Let $M_{N,n}^F$ denote the set of all $N \times n$ matrices over $F$ and let $M_N^F := M_{N,N}^F$.
For $N \in \N$, we set $N^F:=N  \cdot \dim_\R F$ and 
sometimes  identify $M_{N,n}^F$ with $(F^N)^n$ and $\R^{N^F n}$, 
where we consider $F^N$ as the space of numerical column vectors over $F$.
Let $\{e_l\}_{l=1}^N$ denote the standard basis of $F^N$.
The identity matrix of size $N$ is indicated by $I_{N}$ and
the $N \times n$  zero matrix by $0_{N,n}$.
For $Z=(z^m_l)_{1\leq m\leq N, 1\leq l\leq n} \in M_{N,n}^F$, 
its {\it adjoint} $Z^\ast$ (reap.\ {\it transpose} $\T Z$) 
is the $n \times N$ matrix with $(l,m)$-component $(z^m_l)^\ast$ (resp. $z^m_l$).
The \emph{trace} and the \emph{Frobenius norm} of $Z=(z^m_l)_{1\leq m, l\leq N} \in M_{N}^F$ are respectively defined as 
\[
\tr(Z):=\sum_{l=1}^N z^l_l, \qquad
\|Z\|:=\sqrt{\tr(Z^\ast Z)}.
\]
It follows that 
$(\tr (Z))^\ast=\tr (Z^\ast)$ and  $(ZW)^\ast=W^\ast Z^\ast$ for any $Z, W\in M_{N}^F$. 
Define an $F$-valued function $\lr{\cdot}{\cdot}$ on $M_{N,n}^F \times M_{N,n}^F$ by
\[
\lr{Z}{W}:=\tr (Z^\ast W)
\] 
for $Z,W \in M_{N,n}^F$.
Then, $\Re\lr{\cdot}{\cdot}$ is an $\R$-inner product on $M_{N,n}^F$ and 
we have
\[
\|Z-W\|^2
=\|Z\|^2+\|W\|^2- (\tr(Z^\ast W)+\tr(W^\ast Z))
=\|Z\|^2+\|W\|^2- 2 \Re \lr{Z}{W}.
\]
Although $\lr{\cdot}{\cdot}$ is not $F$-bilinear,
a standard argument proves the Cauchy--Schwarz inequality
$\|\lr{Z}{W} \| \leq \|Z\|\|W\|$ for any $Z,W \in M_{N,n}^F$.
\begin{defn}[Unitary and Hermitian matrices]
We say that a matrix $Z \in M_{N}^F$ is {\it unitary} (resp.\ {\it Hermitian}) if we have $Z^\ast Z=Z Z^\ast=I_N$ (resp.\ $Z^\ast=Z$).
\end{defn}
We remark that  $U \in M_{N}^F$ is unitary if and only if $\|Uz\|=\|z\|$ holds for any $z \in F^N$.
Let  $U^F(N)$ denote the group of unitary matrices of size $N$.
For any Hermitian matrix $H\in M_{N}^F$, 
there exist $U \in U^F(N)$ and $\{\sigma_l\}_{l=1}^N \subset \R$ such that 
\[
H=U \cdot \diag(\sigma_1,\ldots,\sigma_N) \cdot U^\ast
\]
 (see \cite{Zhang}*{Corollary 6.2}), where $\diag(\cdots)$ denotes the diagonal matrix.
This decomposition is called the {\it eigen decomposition} and 
$\sigma_l$, $l=1,\dots,N$, are called the {\it eigenvalues} of $H$.
It is easy to check that
$Z^\ast Z$ for $Z \in M_{N,n}^F$ is Hermitian and all its eigenvalues are non-negative.

Let us introduce two matrix decompositions.
\begin{thm}[cf.~\cite{Zhang}*{Theorems 7.1 and 7.2}] \label{thm:decomp}
For any matrix $Z \in M_{N,n}^F$ with $N \geq n$, 
there exist a Hermitian matrix $H \in M_{n}^F$ with non-negative eigenvalues and 
a matrix $Q \in M_{N,n}^F$ such that 
\[
Z=Q\cdot H, \quad Q^\ast Q=I_n.
\]
In addition, there also exist $U \in U^F(N), V \in U^F(n)$ and 
a monotone non-increasing sequence $\{\lambda_l(Z)\}_{l=1}^n$ of non-negative numbers such that 
\begin{gather*}
Z=U \cdot \Lambda \cdot V^\ast, \quad
\Lambda:=
\begin{pmatrix}
\diag(\lambda_1(Z),\ldots,\lambda_n(Z))\\
0_{N-n,n}
\end{pmatrix},\\
Q= U \cdot 
\begin{pmatrix}
I_n \\
0_{N-n,n}
\end{pmatrix}  \cdot V^\ast, 
\quad
H=V \cdot \diag(\lambda_1(Z),\ldots,\lambda_n(Z))  \cdot V^\ast.
\end{gather*}
\end{thm}
The two decompositions $Z=QH$ and $Z=U\Lambda V^\ast$
are called \emph{polar} and \emph{singular value decompositions},
respectively.
Although the two matrix decompositions may not be unique in general, 
$\{\lambda_l(Z)\}_{l=1}^n$ is uniquely determined, which coincides with the positive square root of the eigenvalues of $Z^\ast Z$.
We call $\lambda_l(Z)$, $l=1,\dots,n$, the  {\it singular values} of $Z$. 
In the case of  $\lambda_n(Z)>0$, the polar decomposition is unique
and $Q,H$ are given by
\[
Q=ZH^{-1}, \quad
H=(Z^\ast Z)^{1/2}.
\]
\begin{rem}\label{spectrum}
(1)
Given any $U\in U^F(N)$ and $Z \in M_{N,n}^F$, 
we see that $Z=Q\cdot H$ is a polar decomposition of $Z$
if and only if so is $UZ=(UQ)\cdot H$.\\
(2)
We observe that the maximal singular value of a matrix coincides with its spectrum norm.
The triangle inequality for the spectrum norm implies that 
$\lambda_1(Z+W) \leq \lambda_1(Z)+\lambda_1(W)$.
\end{rem}

\subsection{Stiefel manifold and its quotient space}
\begin{defn}[Stiefel manifold]
For $N,n \in \N$ with $N \geq n$, 
the $(N,n)$-{\it Stiefel manifold} $V_{N,n}^F$ over $F$ is the set of all orthonormal $n$-frames in $F^N$, namely 
\[
V_{N,n}^F=\{ (z_1,\ldots, z_n) \in M_{N,n}^F \ |\ \lr{z_l}{z_m}= \delta_{lm}, \ 1 \leq  l, m \leq n  \},
\]
where $\delta_{lm}$ is the Kronecker delta.
Denote by $\nu^{N,n,F}$ the Haar (or uniform) probability measure on $V_{N,n}^F$.
\end{defn}
Note that $V_{N,N}^F=U^F(N)$ and 	
$V_{N,1}^F$ is identified with the $(N^F-1)$-dimensional Euclidean unit sphere.
%

Let us recall  a characterization  of a Haar measure.
Given any $U \in U^F(N)$,  
we define the map $\mathcal{U}^U_{m}:M_{N,m}^F \to M_{N,m}^F$  
by $\mathcal{U}^U_{m}(Z)=UZ$ for $Z \in M_{N,m}^F$.
The map $\mathcal{U}^U_{m}$ is isometric 
and its inverse map is given by $\mathcal{U}^{U^\ast}_{m}$.
\begin{prop}[cf.~\cite{MS}*{Theorem 1.3}] \label{prop:haar}
A Borel probability measure $\nu$ on $V_{N,n}^F$ coincides with $\nu^{N,n,F}$ 
if and only if $\nu$ is left-invariant under the $U^F(N)$-action, that is,
\[
\nu(\mathcal{U}_n^U(B))=\nu(B)
\]
holds for any Borel set $B \subset V_{N,n}^F$ and $U \in U^F(N)$.
\end{prop}
We consider the $(N,n)$-Stiefel manifold over $F$ as an mm-space $V_{N,n}^F=(V_{N,n}^F, \|\cdot\|,\nu^{N,n,F})$.

The unitary group $U^F(n)$ acts on $M_{N,n}^F$ by right multiplication, that is,
\[
M_{N,n}^F \times U^F(n) \ni  (Z,U)\mapsto ZU \in M_{N,n}^F.
\]
We also consider the \emph{Hopf action} on $M_{N,n}^F$ that is the acton of the unitary group $U^F(1)=\{\;t \in F\ |\ \|t\|=1\;\}$ of size $1$ given by left multiplication,
\[
U^F(1) \times  M_{N,n}^F  \ni  (t,Z)\mapsto tZ \in M_{N,n}^F.
\]
Note that every orbit in $M_{N,n}^F$ of $U^F(n)$ and $U^F(1)$ is closed.
We call the two quotient mm-spaces
\[
G_{N,n}^F := V_{N,n}^F/U^F(n) \quad\text{and}\quad \PV_{N,n}^F := U^F(1)\backslash V_{N,n}^F
\]
the $(N,n)$-\emph{Grassmann  manifold}
and the $(N,n)$-\emph{projective Stiefel manifold over $F$},
respectively.
The $(N,n)$-Stiefel manifold over $F$ with distance multiplied by $\sqrt{N^F - 1}$ is identified with 
\begin{align*}
X_{N,n}^F := \{ (z_1,\ldots, z_n) \in M_{N,n}^F \ |\ \lr{z_l}{z_m}=(N^F-1) \delta_{lm}, \ 1 \leq  l, m \leq n  \}
\end{align*}
with 
the Frobenius norm and 
the Haar probability measure $\mu^{N,n,F}$ on $X_{N,n}^F$.
Note that $\mu^{N,n,F}$ is a unique left-invariant Borel probability measure under the $U^F(N)$-action.
\subsection{Gaussian space}\label{sec:Gauss}
For a positive integer $m$, we denote by  $\gamma^m$ 
the (standard) {\it Gaussian measure on $\R^m$ }, 
which is defined for a Lebesgue measurable set $ B \subset \R^m$ by
\[
\gamma^m(B)=(2\pi)^{-m/2} \int_B \exp \left(-\frac{\|x\|^2}{2}\right) dx.
\]
The mm-space $\Gamma^m:=(\R^m, \|\cdot\|, \gamma^m)$ 
is called the \emph{$m$-dimensional} (\emph{standard}) \emph{Gaussian space}.
By Lemma \ref{lem:obvious}(2), as $m\to\infty$, $\cP_{\Gamma^m}$ converges weakly to
the $\square$-closure of the union of $\cP_{\Gamma^m}$, which we call
the \emph{virtual infinite-dimensional Gaussian space} $\cP_{\Gamma^\infty}$.

Let $l \le N$
and let $\pi_l^N(n):M_{N,n}^F \to M_{l,n}^F$ be the projection defined by
\[
\pi_l^N(n):M_{N,n}^F \ni
\begin{pmatrix}
z^1_1, &\ldots, &z^1_n\\
\vdots & &\vdots\\
z^N_1 &\ldots, &z^N_n
\end{pmatrix}
\mapsto
\begin{pmatrix}
z^1_1, &\ldots, &z^1_n\\
\vdots & &\vdots\\
z^l_1 &\ldots, &z^l_n
\end{pmatrix}
\in M_{l,n}^F.
\]
We set $\pi^N_l := \pi^N_l(1) : F^N \to F^l$.
The projections $\pi^N_l$ and $\pi_l^N(n)$ are both $1$-Lipschitz continuous, preserving Gaussian measures,
and equivariant under the $U^F(1)$-Hopf action and the $U^F(n)$-action, respectively.
We remark that the $U^F(1)$-Hopf action and the $U^F(n)$-action for $n=1$
do not coincide with each other in the case where $F = \H$,
because of the non-commutativity of $\H$.
We have the quotient maps
$\overline{\pi}^N_l : U^F(1)\backslash F^N\to U^F(1)\backslash F^l$ and
$\overline{\pi}^{N}_l(n) : M_{N,n}^F/U^F(n)\to M_{l,n}^F/U^F(n)$,
which are both $1$-Lipschitz continuous by Lemma \ref{lem:equivdomi}.
Note that the actions of $U^F(1), U^F(n)$ on $F^N, M_{N,n}^F$ each preserve
the Gaussian measure.
We denote by $U^F(1)\backslash\Gamma^{N^F}$ and $\Gamma^{N^F n}/U^F(n)$
the quotient mm-spaces of $(F^N,\|\cdot\|,\gamma^{N^F})$ and $(M_{N,n}^F, \|\cdot\|,\gamma^{N^F  n})$
by the $U^F(1)$-Hopf and $U^F(n)$ actions, respectively (see Definition \ref{defn:quotient}).
Then, the sequences $\{U^F(1)\backslash\Gamma^{N^F}\}_{N=1}^\infty$ and
$\{\Gamma^{N^F  n }/U^F(n)\}_{N=1}^\infty$
are both monotone increasing with respect to the Lipschitz order.
Therefore, the associated pyramids converge weakly to the $\Box$-closure of the unions,
\[
\cP_{U^F(1)\backslash\Gamma^{\infty}}
:=
\overline{\bigcup_{N=1}^\infty \cP_{U^F(1)\backslash\Gamma^{N^F}}}^\Box,
\qquad
\cP_{\Gamma^{\infty n}/U^F(n)}:=\overline{ \bigcup_{N=1}^\infty \cP_{\Gamma^{N^F  n }/U^F(n)}}^\Box,
\]
respectively.
We remark that for each positive integer $n$ the quotient mm-space of $(M^F_{N,n},\|\cdot\|,\gamma^{N^F n})$
by the $U^F(1)$-Hopf action is mm-isomorphic to $U^F(1)\backslash \Gamma^{N^F n}$
whose associated pyramid converges weakly to $\cP_{U^F(1)\backslash\Gamma^{\infty}}$
as $N \to \infty$.


Let us close this section with two approximations related to the Gaussian measure.
\begin{prop}[\cite{Watson}*{\S II}] \label{thm:MB}
For any $l,n \in \N$, $\lim_{m \to \infty}\dP(\pi^{m}_l(n)_\# \mu^{m,n,F},\gamma^{l^F n})=0$.
\end{prop}

Proposition \ref{thm:MB} is a generalization of the Maxwell-Boltzmann distribution law.

\begin{prop}[Stirling's approximation] \label{lem:stir}
Let $\Gamma$ be the Gamma function.
There exists a decreasing function $\rho:(0,\infty) \to (0,\infty)$ such that 
\[
\Gamma(x)=\frac{\Gamma(x+1)}{x}=\sqrt{ \frac{2\pi}{x}} \left(\frac{x}{e}\right)^x e^{\rho{(x)}}, 
\quad
\rho(x) \in\left(0, \frac{1}{12x}\right).
\]
\end{prop}

\section{Relation between Gaussian space and Stiefel manifold}
For any $\varepsilon,r>0$ and $m \in \N$, we set   
\[
A^{m}(r)_{\varepsilon}:=B_{\varepsilon r}\left(S^{m-1}(r)\right)
=\{x \in \R^m\ |\ (1-\varepsilon)r <\|x\|<(1+\varepsilon)r\}, \quad
A^m_\varepsilon:=A^{m}(\sqrt{m-1})_{\varepsilon}.
\]
In this section, we first provide a sufficient condition 
for $\{\varepsilon_m\}_{m=1}^\infty$ being $\lim_{m\to\infty}\gamma^m (A^{m}_{\varepsilon_m}) =1$.
Using the sufficient condition, 
we prove that 
the Prohorov distance between $\gamma^{N^F n}$ and $\mu^{N,n,F}$ does not vanish asymptotically
(see Theorem \ref{prok})
although 
$\gamma^{N^F n}$ concentrates around $X_{N,n}^F$ (see Theorem \ref{thm:fullmeas}),
where we regard $\mu^{N,n,F}$ as a probability measure on $\R^{N^F n}$ via the natural embedding  $X_{N,n}^F \subset \R^{N^F n}$.
%

\subsection{Behavior of Gaussian measure} \label{ssec:Gaussian}
Consider the function
\[
g_m(r) :
= \frac{\vol (S^{m-1}(1))}{(2\pi)^{m/2}}r^{m-1} e^{-r^2/2}
=\frac{2^{(2-m)/2}}{\Gamma(m/2)}r^{m-1} e^{-r^2/2}, 
\]
which is the density of the radial distribution of the $m$-dimensional Gaussian measure.
This satisfies 
\begin{equation*}
g_m(r) \leq g_m(\sqrt{m-1})
=
\frac{e^{-\rho(m/2)}}{\sqrt{\pi}} e^{1/2}\left(1-\frac1m\right)^{(m-1)/2}\xrightarrow{m \to \infty}
\frac1{\sqrt{\pi}}
\end{equation*}
by Lemma \ref{lem:stir}.
Moreover, $g_m(\sqrt{m-1})$ is monotone decreasing in $m$.
We directly compute  
\begin{gather}\label{eq:gauss} 
\gamma^{m}(A^m_{\varepsilon})
=
\int_{(1- \varepsilon) \sqrt{m-1}}^{(1+ \varepsilon) \sqrt{m-1}} g_{m}(r) dr
\leq 
\frac{e^{-\rho(m/2)}}{\sqrt{\pi}} e^{1/2}\left(1-\frac1m\right)^{(m-1)/2}\cdot 2\varepsilon\sqrt{m-1}, \\ 
\label{eq:gau}
1-\gamma^{m}(A^m_{\varepsilon})
=
\int_{0}^{(1- \varepsilon_m) \sqrt{m-1} } g_{m}(r) dr+\int_{(1+ \varepsilon_m)\sqrt{m-1}}^\infty  g_{m}(r)  dr.
\end{gather}
We provide a sufficient condition for $\{\varepsilon_m\}_{m=1}^\infty$ 
such that  $\lim_{m \to 1} \gamma^m (A^{m}_{\varepsilon_m})=1$,
based on the idea of \cite[Lemma 6.1]{Sy:mmlim}.
%
\begin{lem}\label{lem:zero}
The zero $t_0$ of the function given by 
\[
G(t):=e^{-t^2/2}-\int_t^\infty e^{-s^2/2} ds, \qquad t \in \R.
\]
is unique and lies in $(0,1)$.
\end{lem}
\begin{proof}
The lemma follows from 
the intermediate value theorem and the following properties
\[
G'(t)=e^{-t^2/2}(-t+1), \quad
G(0)=1-\sqrt{\frac{\pi}{2}}<0=\lim_{t\to \infty} G(t) <G(1). 
\]
\end{proof}
\begin{rem}\label{rem:zero}
If we set 
\begin{align*}
G_m^+(r)
&:=e^{-(r\sqrt{m-1})^2/2}-\int_{(1+r)\sqrt{m-1}}^\infty e^{-(\sqrt{m-1}-s)^2/2} ds
=e^{-(r\sqrt{m-1})^2/2}-\int_{r\sqrt{m-1}}^\infty e^{-s^2/2} ds,\\
G_m^-(r)
&:=e^{-(r\sqrt{m-1})^2/2}-\int_0^{(1-r)\sqrt{m-1}} e^{-(\sqrt{m-1}-s)^2/2} ds
=e^{-(r\sqrt{m-1})^2/2}+\int_{\sqrt{m-1}}^{r\sqrt{m-1}}e^{-s^2/2} ds,
\end{align*}
then $G_m^-(r)>G_m^+(r)=G(r\sqrt{m-1})$ always holds and 
$G_m^+(r) \geq 0$ if $r\sqrt{m-1} \geq t_0$.
\end{rem}
%
\begin{lem}\label{lem:gauss}
If $\lim_{m \to \infty} \varepsilon_m \sqrt{m-1} =\infty$, 
then $\lim_{m \to \infty} \gamma^m(A^m_{\varepsilon_m})=1$.
\end{lem}
\begin{proof}
Since we observe 
\[
\frac{d}{dr}\log g_m(r)\big|_{r=\sqrt{m-1}}=0, \quad
\frac{d^2}{dr^2}\log g_m(r)
= -1 -\frac{m-1}{r^2} \leq -1,
\]
it holds for any  $r>0$ that
\begin{align*}
\log g_m(r)-\log g_m(\sqrt{m-1}) 
&=\int_r^{\sqrt{m-1}}\int_{t}^{\sqrt{m-1}} \left( \frac{d^2}{ds^2}\log g_m(s) \right) ds dt \\
&
\leq \int_r^{\sqrt{m-1}}\int_{t}^{\sqrt{m-1}} (-1) ds dt 
=-\frac12 (\sqrt{m-1}-r)^2, 
\end{align*}
providing
$g_m(r) \leq g_m(\sqrt{m-1}) \exp(-(\sqrt{m-1}-r)^2/2 )$.
We deduce from  this and Remark \ref{rem:zero} that 
if $\varepsilon_m \sqrt{m-1} >t_0$,
then we see 
\begin{align*}
\int_0^{(1- \varepsilon_m) \sqrt{m-1} } g_{m}(r) dr
&
\leq  g_{m}(\sqrt{m-1})  \int_0^{(1- \varepsilon_m) \sqrt{m-1} } e^{- (\sqrt{m-1}-r)^2/2} dr  \\
&
\leq g_{m}(\sqrt{m-1}) e^{-(\varepsilon_m \sqrt{m-1})^2/2},\\
\int_{(1+ \varepsilon_m) \sqrt{m-1} }^\infty g_{m}(r) dr
&
\leq  g_{m}(\sqrt{m-1})  \int_{(1+ \varepsilon_m) \sqrt{m-1} }^\infty e^{- (\sqrt{m-1}-r)^2/2} dr  \\
&\leq g_{m}(\sqrt{m-1}) e^{-(\varepsilon_m \sqrt{m-1})^2/2}. 
\end{align*}
This with \eqref{eq:gau} leads to
\begin{align*}
\liminf_{m \to \infty} \left( 1-\gamma^{m}(A^m_{\varepsilon_m}) \right)
\leq 
\liminf_{m \to \infty}
\left\{g_{m}(\sqrt{m-1}) \cdot 2e^{-(\varepsilon_m \sqrt{m-1})^2/2}\right\}
=0.
\end{align*}
The proof is completed.
\end{proof}
\subsection{Prohorov distance between $\gamma^{N^F n}$ and $ \mu^{N,n_N,F}$}
\begin{proof}[Proof of Theorem \ref{prok}]
By the definition of the Prohorov distance, 
it holds for any
$D_{N}>\dP(\gamma^{N^F n_N}, \mu^{N,n_N,F})$ that 
\begin{align*}
1-D_N&=
\nu^{N,n_N,F}(X_{N,n_N}^{F})-D_N \\
&\leq 
\gamma^{N^{F}   n_N}(B_{D_N}(X_{N,n_N}^{F})) 
\leq 
\left( \gamma^{N^{F}}( B_{D_N}(X_{N,1}^{F})) \right)^{n_N} \\ 
&\leq \left\{
\frac{e^{-\rho(N^F/2)}}{\sqrt{\pi}} e^{1/2}\left(1-\frac1{N^F}\right)^{(N^F-1)/2}\cdot 2D_N \right\}^{n_N},
\end{align*}
where we use the property 
that $B_{D_N} (X_{N,n_N}^{F})\subset (B_{D_N}(X_{N,1}^{F}))^{n_N}$ 
in the second inequality
and the last inequality follows from \eqref{eq:gauss}.
Since
\begin{align*}
\log\left\{e^{1/2}\left(1-\frac1{m}\right)^{(m-1)/2}\right\}^m
= \frac{m}2\left\{1-(m-1)\left(\frac1{m}+\frac1{2m^2}+o\left(\frac1{m^2}\right)\right) \right\} = \frac{1}{4}+o(1)
\end{align*}
and $e(1-1/m)^{m-1}>1$ for any $m \in \N$,
the property $n_N \leq N  \leq N^{F}$ implies
\begin{align*}
1  
\leq \liminf_{N\to\infty} 
\left\{D_N+\gamma^{N^{F} n_N} (B_{D_N} (X_{N,n_N}^{F})) \right\}
\leq
\liminf_{N\to\infty} 
\left\{D_N+ \left(\frac{2 D_N}{\sqrt{\pi}}\right)^{n_N} e^{1/4} \right\},
\end{align*}
providing $\liminf_{N\to\infty} D_N>0$.
Since $D_{N}>\dP(\gamma^{N^F n_N}, \mu^{N,n_N,F})$ is arbitrary, 
we conclude  $\liminf_{N\to\infty} \dP(\gamma^{N^F n_N}, \mu^{N,n_N,F})>0$.
This completes the proof.
\end{proof}

\subsection{$\gamma^{N^F n}$ concentrates around $X_{N,n}^{F}$}
Given a function $\theta:(0,\infty) \to (0,\infty)$ and $\varepsilon>0$,
we define the \emph{$(\varepsilon, \theta)$-approximation space $X_{N,n,\varepsilon,\theta}^F$ 
of $X_{N,n}^{F}$} as       
\begin{align*}
X_{N,n,\varepsilon,\theta}^F
:=\left\{(z_1,\ldots,z_n) \in M_{N,n}^F \biggm| \|z_l\| \in A^{N^F}_\varepsilon,
                            \left\|\lr{\frac{z_l}{\|z_l\|}}{\frac{z_m}{\|z_m\|}}\right\|< \theta(\varepsilon), \ 1 \leq  l<m \leq n\right\},
\end{align*}
where we identify $F^N$ with $\R^{N^F}$.

The purpose of this subsection is  to prove the following theorem.
\begin{thm}\label{thm:fullmeas}
Suppose that  a sequence $\{n_N\}_{N=1}^\infty \subset \N$  satisfies 
\begin{equation}\label{eq:ass}
\sup_N \left( 2\log n_N - \frac{a'}{4} \cdot \left(\frac{N}{n_N}\right)^{1-a} \right) < \infty
\quad 
\text{\ for some } a, a' \in (0,1)
\end{equation}
and put 
\begin{gather*}
p_N:=\log_{(N-1)} n_N,\quad 
a_N:=\frac{a}{2}(1-p_N), \quad
\varepsilon_N:=(N-1)^{-a_N}.
\end{gather*}
If we choose 
a function $\theta:(0,\infty) \to (0,\infty)$ as 
\begin{equation}\label{theta}
\theta(\varepsilon):= \left(\frac{5\varepsilon^2(1+\varepsilon)}{(1-\varepsilon)+5\varepsilon^2(1+\varepsilon)}\right)^{1/2},
\end{equation}
then we have 
\[
\lim_{N \to \infty}\varepsilon_N=0, \quad
\lim_{N \to \infty}\gamma^{N^F n_N}(X^F_{N,n_N,\varepsilon_{N},\theta}) = 1.
\]
In the case that
\begin{align}\label{eq:condi}
\sup_N \left( 2\log n_N - \frac{a'}{4} \sqrt{\frac{N}{n_N^3}} \right) < \infty
\quad 
\text{\ for some } a' \in (0,1),
\end{align}
the statement also holds true if we set $a_N:=(1+p_N)/4$.
\end{thm}
Note that \eqref{eq:condi} coincides with ($*$) in the introduction.
In the rest of this section, we always suppose \eqref{eq:ass} or \eqref{eq:condi}.
Note that 
\eqref{eq:ass} follows from \eqref{eq:condi} (see Remark \ref{rem:condi}\eqref{c1}\eqref{c3}) and 
$(1+p_N)/2 \leq 1-p_N$ is equivalent to $p_N \leq 1/3$.
It turns out that 
\eqref{eq:ass} (resp. \eqref{eq:condi}) is equivalent to
\[
\sup_N \left\{ n_N^2  \exp\left(-\frac{a'}{4}\left(\frac{N-1}{n_N}\right)^{1-a}\right)\right\}<\infty
\ 
\left(\text{resp.} 
\sup_N \left\{n_N^2  \exp\left(-\frac{a'}{4}\sqrt{\frac{N-1}{n_N^3}}\right)\right\}<\infty
\right),
\]
which yields for any $r>0$ that 
\begin{gather}\label{eq:lim}
\lim_{N\to \infty} (N-1)^{r(p_N-1)}=\lim_{N\to \infty} \left(\frac{n_N}{N-1}\right)^r=0 \quad
\left(
\text{resp.}\ 
\lim_{N\to \infty} (N-1)^{r(3p_N-1)}=0
\right).
\end{gather}
We may assume that $n_N \leq N-1$ without loss of generality.
We in addition assume that $p_N \le 1/3$ if \eqref{eq:condi} is satisfied.
In this case, $n_N/(N-1) \leq \varepsilon_N^2 \leq 1$ holds for any $N \in \N$.
To prove Theorem \ref{thm:fullmeas}, we define $b_l, \varepsilon_{N,l}, T_{N,l}$ for $1 \leq l \leq n_N-1$ by
\begin{align*}
b_N:&=1-\varepsilon_N, \\
\varepsilon_{N,l}
:&=1-(1-b_N\varepsilon_N) \sqrt{\frac{N-1}{(N-l)-1}} \\
 &=\frac{1}{\sqrt{(N-l)-1}}\left(b_N\varepsilon_N \sqrt{N-1}-\frac{l}{\sqrt{N-1}+\sqrt{(N-l)-1}}\right),\\
T_{N,l}
:&=\left\{\frac{(N-l)-1}{l}(\varepsilon_{N}-\varepsilon_{N,l})(\varepsilon_{N}+\varepsilon_{N,l}+2) \right\}^{1/2}\\
&=\left\{\frac{(N-l)-1}{l} ((1+\varepsilon_{N})^2-(1+\varepsilon_{N,l})^2 )\right\}^{1/2}.
\end{align*}
%
\begin{lem}\label{lem:limit}
The sequence $\{\varepsilon_{N,l}\}_{l=1}^{n_N-1}$ is monotone decreasing and $\varepsilon_{N,1}<\varepsilon_{N}$.
For any $1 \leq l \leq n_N-1$, we have
\begin{gather*}
\lim_{N \to \infty}\varepsilon_{N}=0, \
\lim_{N \to \infty}\varepsilon_N \sqrt{N^F-1}=\infty, \
\lim_{N \to \infty}\varepsilon_{N,l}\sqrt{(N-l)^F-1} =\infty, \ 
\lim_{N \to \infty}T_{N,l}=\infty.
\end{gather*}
\end{lem}
\begin{proof}
%
The first statement follows from a direct computation as
\begin{align*}
&\varepsilon_{N,l-1}-\varepsilon_{N,l}
=
(1-b_N\varepsilon_N)\sqrt{N-1}\left(\frac{1}{\sqrt{(N-l)-1}}-\frac{1}{\sqrt{N-l}}\right)
> 0, \\
&\varepsilon_{N}-\varepsilon_{N,l}
=(1-b_N\varepsilon_N) \left(\sqrt{\frac{N-1}{(N-l)-1}}-1\right)+\varepsilon_N(1-b_N)> 0.
%
\end{align*}
It follows from \eqref{eq:lim} that $\lim_{N \to \infty}\varepsilon_{N}=0$.
We  moreover have
\begin{gather*}
\log_{(N-1)}\varepsilon_N \sqrt{N^F-1}
\geq
\log_{(N-1)}\varepsilon_N \sqrt{N-1}
=
\frac12-a_N
=
\begin{cases}
\displaystyle
\frac{1-a+ap_N}{2}
\geq\frac{1-a}{2}
&\text{for \eqref{eq:ass}},  \\
\displaystyle
\frac{1-p_N}{4}
\geq \frac16 
&\text{for \eqref{eq:condi}}, 
\end{cases}\\
\log_{(N-1)} \frac{\varepsilon_N(N-1)}{n_N}
\geq
\log_{(N-1)} \frac{\varepsilon_N^2(N-1)}{n_N}
= 1-2a_N-p_N
=
\begin{cases}
\displaystyle
 (1-a)(1-p_N)& \text{for \eqref{eq:ass}},  \\
\displaystyle
\frac{1-3p_N}{2}
& \text{for \eqref{eq:condi}}.
\end{cases}
\end{gather*}
These and \eqref{eq:lim} together yield that 
\[
\lim_{N \to \infty} \varepsilon_N \sqrt{N-1}=\infty, 
\quad
\lim_{N \to \infty} \frac{\varepsilon_N^2(N-1)}{n_N}=\infty,
\quad
\lim_{N \to \infty} \frac{n_N}{\varepsilon_N(N-1)}=0.
\]
Since we observe that 
\begin{align}\label{en}\notag
\varepsilon_{N,l}\sqrt{(N-l)^F-1}
%
&\geq
\varepsilon_{N,n_N-1} \sqrt{N-n_N}  \\ 
&=\varepsilon_N \sqrt{N-1} \left(b_N-\frac{\sqrt{N-1}}{\sqrt{N-1}+\sqrt{N-n_N}}\cdot \frac{n_N-1}{\varepsilon_N(N-1)} \right),\\ \notag
T_{N,l}^2
&=\frac{(N-l)-1}{l}  (\varepsilon_N-\varepsilon_{N,l})  (\varepsilon_{N}+\varepsilon_{N,l}+2) \\ \label{tn}
&\geq \frac{(N-l)-1}{l} \cdot \varepsilon_N(1-b_N) \cdot 2
\geq 2  \frac{\varepsilon_N^2(N-1)}{n_N} \cdot \frac{n_N}{n_N-1} \frac{N-n_N}{N-1},
\end{align}
we have 
$\lim_{N \to \infty}\varepsilon_{N,l}\sqrt{(N-l)^F-1} =\infty$ and  
$\lim_{N \to \infty}T_{N,l}=\infty$.
This completes the proof.
\end{proof}
\begin{cor}\label{cor:en}
For all sufficiently large $N$, we have
\[
\varepsilon_{N,l} \sqrt{(N-l)^F-1} \geq  a' \varepsilon_{N} \sqrt{N-1}, \quad
1 \leq l \leq  n_{N}-1. 
\]
\end{cor}
\begin{proof}
By \eqref{en}, 
we have 
\[
\varepsilon_{N,l} \sqrt{(N-l)^F-1}
\geq
\varepsilon_N \sqrt{N-1} \left(b_N-\frac{\sqrt{N-1}}{\sqrt{N-1}+\sqrt{N-n_N}}\cdot \frac{n_N-1}{\varepsilon_N(N-1)} \right).
\]
Then the corollary follows from the fact
\[
\lim_{N \to \infty} \left(b_N-\frac{\sqrt{N-1}}{\sqrt{N-1}+\sqrt{N-n_N}}\cdot \frac{n_N-1}{\varepsilon_N(N-1)} \right)=1.
\]
\end{proof}
%
To compute $\gamma^{N^F n }(X_{N,n,\varepsilon,\theta}^{F})$, 
let us regard  $U^F(N)$ as a subgroup of $U^{\R}(N^F)$ (see Lemma \ref{lem:unit})
and put 
\[
A^{N,F}_{\varepsilon, \theta}[z_1,\ldots,z_l]:=
\left\{z_{l+1} \in A^{N^F }_{\varepsilon} \ \bigg|\  \left\| \lr{\frac{z_m}{\|z_m\|}}{\frac{z_{l+1}}{\|z_{l+1}\|}} \right\|
< \theta(\varepsilon),\ 1 \leq  m \leq l \right\}
\]
for any  $(z_1,\ldots,z_n) \in  X_{N,n,\varepsilon,\theta}^{F}$ and $1 \le l \le n-1$.
We then have 
\[
\gamma^{N^F n }(X_{N,n,\varepsilon,\theta}^{F})
=\int_{z_1 \in A^{N^F }_{\varepsilon}} 
 \int_{z_{2} \in A^{N,F}_{\varepsilon,\theta }[z_1]} \cdots 
 \int_{z_n \in A^{N,{F}}_{\varepsilon,\theta}[z_1,\ldots,z_{n-1}]}
d\gamma^{N^F}(z_{n}) \cdots d\gamma^{N^F }(z_2)  d\gamma^{N^F }(z_1).
\]
%
%
Choose $U \in U^F(N)$ satisfying that $\lr{e_m}{Uz_l}=0$ 
for any pair $(m,l)$ with  $m>l$.
According to the $U^\R(N^F )$-invariance of $\gamma^{N^F }$, 
it holds that 
\[
\gamma^{N^F } (A^{N,F}_{\varepsilon,\theta}[z_1,\ldots,z_l])
=\gamma^{N^F } (\mathcal{U}^{U}_1( A^{N,F}_{\varepsilon,\theta}[z_1,\ldots,z_l]))
=\gamma^{N^F } (A^{N,F}_{\varepsilon,\theta}[Uz_1,\ldots,Uz_l]).
\]
\begin{lem}\label{lem:subset}
For any $1 \leq l \leq n_N-1$, 
if we choose $\theta$ as in \eqref{theta}, then we have 
\[
\left(B^{F}(T_{N,l})\right)^{l} \times A^{(N-l)^{F}}_{\varepsilon_{N,l}} \subset A^{N,F}_{\varepsilon_N,\theta}[Uz_1,\ldots,Uz_l], \quad
B^{F}(T):=\{ z \in F \ |\ \|z\|<T\}.
\]
\end{lem}
\begin{proof}
Given any
\[
w:=(w_1,\ldots, w_l)\in \left(B^{F}(T_{N,l}) \right)^{l} \subset F^l, \quad
\zeta \in A^{(N-l)^{F}}_{\varepsilon_{N,l}} \subset F^{N-l},
\]
we prove $z_{l+1}:=(w,\zeta) \in A^{N^F }_{\varepsilon_N}$
by the following computations:
\begin{align*}
\|z_{l+1}\|^2
&=\|w\|^2+\|\zeta\|^2
< l (T_{N,l})^2+ (1+\varepsilon_{N,l})^2((N-l)^F -1) \\
& = \left\{ \frac{l (T_{N,l} )^2}{(N-l)^F -1}+ (1+\varepsilon_{N,l} )^2 \right\}((N-l)^F -1) \\
& \leq (1+\varepsilon_N  )^2 ((N-l)^F -1)
<(1+\varepsilon_N  )^2(N^F -1), \\
\|z_{l+1}\|
&\geq \|\zeta\|
> (1-\varepsilon_{N,l}) \sqrt{(N-l)^{F}-1}
=(1-b_n\varepsilon_N) \sqrt{\frac{N-1}{(N-l)-1}} \cdot \sqrt{(N-l)^{F}-1}\\
&\geq(1-b_n\varepsilon_N) \sqrt{N^F-1}
>(1-\varepsilon_N) \sqrt{N^F -1}.
\end{align*}
Using the fact that $n_N/(N-1) \leq \varepsilon_N^2 \leq 1$, 
we find for $1 \leq l \leq  n_N-1$ that 
\begin{align*}
\frac{\varepsilon_N-\varepsilon_{N,l}}{1-\varepsilon_{N,l}}
&=\frac{(1-b_N\varepsilon_N) \left(\sqrt{\frac{N-l}{(N-l)-1}}-1\right)+\varepsilon_N(1-b_N)}{(1-b_N\varepsilon_N) \sqrt{\frac{N-1}{(N-l)-1}}} \\ \notag
&=\left(1-\sqrt{1-\frac{l}{N-1}}\right)+\frac{\varepsilon_N(1-b_N)}{1-b_N\varepsilon_N} \sqrt{1-\frac{l}{N-1}}\\ 
&
<\frac{l}{N-1} +\frac{\varepsilon_N (1-b_N)}{1-b_N\varepsilon_N}
<\frac{n_N}{N-1}+\frac43\varepsilon_N(1-b_N)
\leq\frac73\varepsilon_N^2,\\ 
\frac{l T_{N,l}^2}{ (1-\varepsilon_{N,l} )^2 ((N-l) -1) } 
&=\frac{ (\varepsilon_N -\varepsilon_{N,l} )(\varepsilon_N +\varepsilon_{N,l} +2)}{(1-\varepsilon_{N,l} )^2}\\ 
&<  \frac{7\varepsilon_N^2( \varepsilon_N +\varepsilon_{N,l} +2)}{3(1-\varepsilon_{N,l} )}  
<  \frac{ 5\varepsilon_N^2(1+\varepsilon_N )}{1-\varepsilon_{N} }
= \frac{\theta(\varepsilon_N)^2}{1-\theta(\varepsilon_N)^2}.  
\end{align*}
This implies 
\begin{align*}
\|w\|^2
<l T_{N,l}^2
&<\frac{\theta(\varepsilon_N )^2}{1-\theta(\varepsilon_N)^2}   (1-\varepsilon_{N,l})^2( (N-l)-1) \\
&\leq 
\frac{\theta(\varepsilon_N )^2}{1-\theta(\varepsilon_N)^2}   (1-\varepsilon_{N,l})^2( (N-l)^F-1) 
<\frac{\theta(\varepsilon_N )^2}{1-\theta(\varepsilon_N )^2} \|\zeta\|^2,
\end{align*}
which leads to   
$\|w\|^2< \theta(\varepsilon_N)^2 (\|w\|^2+\|\zeta\|^2)=\theta(\varepsilon_N)^2 \|z_{l+1}\|^2$. 
By the Cauchy--Schwarz inequality, for any $1 \leq l \leq  n_N-1$,
\begin{align}\label{angle}
 \left\|\lr{\frac{Uz_m}{\|Uz_m\|}}{\frac{z_{l+1}}{\|z_{l+1}\|}}\right\|
=\left\|  \lr{\frac{Uz_m}{\|Uz_m\|}}{\frac{w}{\|z_{l+1}\|}}\right\|
\leq \frac{\|w\|}{\|z_{l+1}\|}< \theta(\varepsilon_N).
\end{align}
We thus obtain $z_{l+1} \in A^{N,F}_{\varepsilon_N,\theta}[Uz_1,\ldots,Uz_l]$.
This completes the proof.
\end{proof}
\begin{lem}\label{lem:full}
For $0 \leq l \leq  n_N-1$, we set 
\[
v_{N,l}^F:=
\begin{cases}
1-\gamma^{N^F} \left(A^{N^F}_{\varepsilon_N}\right) &\text{\ if $l=0$}, \\
1-\gamma^{N^F} \left( B^F(T_{N,l})^l\times A^{(N-l)^{F}}_{\varepsilon_{N,l}}\right) &\text{\ if $l\neq0$},
\end{cases}\quad
v_N^F:=\max\{v_{N,l}^F \ |\ 0 \leq l \leq n_N-1 \}. 
\]
We then have $\lim_{N \to \infty} n_N v_N^F =0$.
\end{lem}
\begin{proof}
Putting
\[
\alpha_{N,l}^F:=
\begin{cases}
\gamma^{N^F}(A_{\varepsilon_N}^{N^F}) &\text{if $l = 0$},\\
\gamma^{(N-l)^{F}}(A^{(N-l)^{F}}_{\varepsilon_{N,l}}) &\text{if $l \neq 0$},
\end{cases}
\qquad
\beta_{N,l}^F:=
\begin{cases}
0  &\text{\ if $l=0$}, \\
1-\gamma^{1^F}( B^{F}(T_{N,l}) ) &\text{\ if $l\neq0$},
\end{cases}
\]
we see 
\[
v_{N,l}^F=(1-\alpha_{N,l}^F) \left(1-\beta_{N,l}^F\right)^l+ 1-\left( 1-\beta_{N,l}^F\right)^l
\leq (1-\alpha_{N,l}^F)+l \beta_{N,l}^F 
\]
and it suffices to prove 
\[
\lim_{N \to \infty} n_N (1-\alpha_{N,l}^F) =0, 
\quad
\lim_{N \to \infty} n_N^2 \cdot \beta_{N,l}^F =0
\]
for any $ 0\leq l \leq  n_N-1$.
Since we have 
$\lim_{N \to \infty} \alpha_{N,l}^F=1, \lim_{N \to \infty} \beta_{N,l}^F=0$
by Lemmas \ref{lem:gauss} and \ref{lem:limit}, 
the lemma holds true if $\sup_N n_N <\infty$.

We consider the case where $\lim_{N \to \infty} n_N=\infty$.
Set $m:=(N-l)^F$.
By \eqref{eq:gau} and the proof of Lemma \ref{lem:gauss},
we have
\begin{align*}
1-\alpha_{N,l}^F
&
\leq
 g_{m}(\sqrt{m-1})\cdot 2 e^{-(\varepsilon_{N,l} \sqrt{m-1})^2/2}.
\end{align*}
For large enough $N$, Corollary \ref{cor:en} implies 
\[
(\varepsilon_{N,l}  \sqrt{m-1})^2
\geq 
a'^2 \cdot \varepsilon_{N}^2 \cdot( N-1)
=
a'^2 (N-1)^{-2a_N+1}
\geq a'^2 n_N
\]
and $m=(N-l)^F \to \infty$ as $N \to \infty$, 
so that we have 
\begin{align*}
n_N(1-\alpha_{N,l}^F)
&
\leq
\left\{
g_{m}(\sqrt{m-1}) \right\}
\cdot 2 (n_N e^{-a'^2n_N})
\xrightarrow{N \to \infty}
0. 
\end{align*}

Let us show $\lim_{N \to \infty} n_N^2 \beta_{N,l}^F =0$.
Assume that $N$ is so large that 
\[
\frac{N-n_{N}}{n_{N}-1}
\geq a' \cdot \frac{N-1}{n_N}.
\]
If we set
\[
 R_N:=
\sqrt{\frac{a'(N-1)}{2n_N}}\varepsilon_N
=\sqrt{\frac{a'}{2}}(N-1)^{(1-p_N-2a_N)/2}, 
\]
which diverges to infinity as $N \to \infty$,
then we observe from \eqref{tn} that 
\begin{align*}
\frac{T_{N,l}^2}{4}
\geq \frac14 \cdot2 \varepsilon_N^2 \cdot \frac{N-n_{N}}{n_{N}-1}
\geq R_N^2.
\end{align*}
Hence we conclude that 
\begin{align*}
B^F(T_{N,l}) \supset 
(B^\R(T_{N,l}/2))^{\dim_\R F}
\supset (B^\R(R_N))^{\dim_\R F},
\end{align*}
which implies that for $\beta:=\gamma^1(B^\R(R_N)) \in[0,1]$,
\[
\beta_{N,l}^F 
\leq 
1-\left(\gamma^1(B^\R(R_N))\right)^{\dim_\R F}
\leq 
1-\beta^4
=(1+\beta^2)(1+\beta)(1-\beta) 
\leq 4(1-\gamma^1(B^\R(R_N))), 
\]
\begin{align*}
1-\gamma^1(B^\R(R_N))
&=
\sqrt{\frac{2}{\pi}} \int_{R_N}^\infty e^{-r^2/2} dr
=\frac{1}{\sqrt{\pi}} \int_{R_N^2/2}^\infty  e^{-t} \frac{1}{\sqrt{t}} dt \\
&=-\frac{1}{\sqrt{\pi}} e^{-t} \frac{1}{\sqrt{t}}\bigg|_{t=R_N^2/2}^{t=\infty}
+\frac{1}{\sqrt{\pi}} \int_{R_N^2/2}^\infty e^{-t}  \frac{-1}{2t\sqrt{t}}
\leq \sqrt{\frac{2}{\pi}}\frac1{R_N}e^{-R_N^2/2}. 
%
\end{align*} 
If the assumption \eqref{eq:ass} holds, 
then we see 
\[
n_N^2 \cdot \beta_{N,l}^F
\leq 
4n_N^2 (1-\gamma^1(B^\R(R_N)))
\leq 
4 \sqrt{\frac{2}{\pi}} \frac{1}{R_N} n_N^2 \exp\left(-\frac{a'}{4}\left(\frac{N-1}{n_{N }}\right)^{1-a}\right)
\xrightarrow{N \to \infty}
0.
\]
Under the assumption \eqref{eq:condi}, we find that
\[
n_N^2 \cdot \beta_{N,l}^F
\leq 
4n_N^2 (1-\gamma^1(B^\R(R_N)))
\leq 
4 \sqrt{\frac{2}{\pi}} \frac{1}{R_N} n_N^2 \exp\left(-\frac{a'}{4}\sqrt{\frac{N-1}{n_{N }^3}}\right)
\xrightarrow{N \to \infty}
0.
\]
This completes the proof.
\end{proof}

\begin{proof}[Proof of Theorem \ref{thm:fullmeas}]
We already observe $\lim_{N \to \infty} \varepsilon_N=0$ in Lemma \ref{lem:limit}. 
We apply Lemma \ref{lem:subset} to have 
\begin{align*}
&\gamma^{N^F n_N}(X_{N,n_N,\varepsilon_N ,\theta}^{F})\\
=&
\int_{z_1 \in A^{N^F }_{\varepsilon_N }}\cdots \int_{z_{n_N-1} \in A^{N,F}_{\varepsilon_N }[z_1,\ldots,z_{n_N-2}]} 
\gamma^{N^F } (A^{N,F}_{\varepsilon_N ,\theta}[z_1,\ldots,z_{n_N-1}]) d\gamma^{N^F }(z_{n_N-1}) \cdots  d\gamma^{N^F }(z_1) 
\\
\geq &
\int_{z_1 \in A^{N^F }_{\varepsilon_N }}\cdots \int_{z_{n_N-2} \in A^{N,F}_{\varepsilon_N ,\theta}[z_1,\ldots,z_{n_N-3}]} \
 \gamma^{N^F } (A^{n,F}_{\varepsilon_N ,\theta}[z_1,\ldots, z_{n_N-2}])
 d\gamma^{N^F }(z_{n_N-2}) \cdots d\gamma^{N^F }(z_1) \\
&\quad \times
 \left( \gamma^{1^F}( B^{F}(T_{N,n_N-1}^{F}) )\right)^{n_N-1}\times \gamma^{(N-(n_N-1))^{F}}(A^{(N -(n_N-1))^{F}}_{\varepsilon_{N,n_N-1}})\\
\geq&
\gamma^n(A^{N^F }_{\varepsilon_N }) \times
\prod_{l=1}^{n_N-1}\left\{ \left( \gamma^{1^F}( B^{F}(T_{N,l}^{F}) ) \right)^l\times 
                               \gamma^{(N-l)^{F}}(A^{(N-l)^{F}}_{\varepsilon_{N,l}})  \right\}
=
\prod_{l=0}^{n_N-1} (1-v_{N,l}^F) \geq (1-v_N^F)^{n_N}.                               
\end{align*}
By Lemma \ref{lem:full},
we have 
$\lim_{N \to\infty} (1-v_N^F)^{n_N}=1$.
This completes the proof of the theorem.
\end{proof}
\begin{rem}\label{rem:angle}
If we choose a function $\theta :(0,\infty) \to (0,\infty)$ satisfying 
\[
\lim_{N \to \infty} \gamma^{N^F n_N} (X^F_{N,n_N,\varepsilon_{N}, \theta})=1
\]
with the use of Lemma \ref{lem:subset}, 
then by \eqref{angle},
the angle $\theta$ is required to satisfy
\begin{align*}
\theta(\varepsilon_N )^2 
&> \sup\left\{  \frac{\|w\|^2}{\|(w,\zeta)\|^2}\ \Bigg|\ 
w\in \left(B^{F}(T_{N,l})\right)^{l},  
\zeta \in A^{(N-l)^{F}}_{\varepsilon_{N,l} }
\right\}\\
&=\frac{lT_{N,l}^2}{l T_{N,l}^2+ (1-\varepsilon_{N,l} )^2((N-l)^F-1)}
=\left(1+\frac{(1-\varepsilon_{N,l} )^2((N-l)^F-1)}{lT_{N,l}^2}\right)^{-1}.
\end{align*}
\end{rem}
\begin{rem}\label{rem:condi}
Let us comment on the conditions \eqref{eq:ass} and \eqref{eq:condi}.
\begin{enumerate}[(a)]
\item\label{c1}
If \eqref{eq:ass} (resp.\ \eqref{eq:condi}) holds, 
we then have $\lim_{N \to \infty}{n_N}/N^{p_\ast}=0$ 
for $p_\ast=1$ (resp.\ $p_\ast=1/3$),  
however the converse does not hold in general.
Such an example is 
\[
n_N:=\left[\left(\frac{N}{\log N}\right)^{p_\ast} +1 \right],
\]
where $[x]$  is the largest integer not greater than $x$.\\
\item \label{c3}
For $p>0$, 
we see that the condition $n_N=O(N^p)$
implies $\lim_{N \to \infty} p_N \leq p$.
If moreover $p<1$, then \eqref{eq:ass} is true,
because there exists $c>0$ such that 
 \[
\sup_{N}\left( 2\log n_N -\frac{a'}{4} \left(\frac{N}{n_N}\right)^{1-a} \right)
\leq 
\sup_{N}\left(2p\log N +2  \log c-\frac{a'}{4} \cdot c^{a-1} N^{(1-p)(1-a)} \right)<\infty
\]
holds for any $a,a' \in (0,1)$.
Similarly, if $p<1/3$, then \eqref{eq:condi} is true. 
\item
If $n_N=o(N (\log N)^{-A})$ for some $A>1$ ,
then \eqref{eq:ass} holds true with $a=(A-1)/A$ and any $a'\in (0,1)$, 
because there exists $N_0 \in \N$ such that 
\[
\frac{n_{N}}{N} (\log N)^A \leq \left(\frac{a'}{4} \times \frac12\right)^{A}
\]
for any $N \geq N_0$, implying
 \[
\sup_{N \geq N_0}\left( 2\log n_N -\frac{a'}{4} \left(\frac{N}{n_N}\right)^{1-a} \right)
\leq 
2\sup_{N\geq N_0}\left( \log N- A\log \log N + A \log\frac{a'}{8}  -\log N \right)<\infty.
\]
It is impossible to reduce the condition $A>1$ to $A=1$, namely $n_N=o(N (\log N)^{-1})$.
For example, if we put
\[
n_N:=\left[\frac{N}{\log N \log \log N} +1 \right],
\]
then this satisfies $n_N=o(N (\log N)^{-1})$,
but does not satisfy \eqref{eq:ass}. 
\item
We similarly derive \eqref{eq:condi} from $n_N=o(N^{1/3} (\log N)^{-2/3} )$.
\end{enumerate}
\end{rem}
\section{Proof of main theorems} 
\subsection{Strategy of the proof}
The idea of the proof of Theorems \ref{thm:Stiefel} and \ref{thm:Gr-pS}
is based on that in the case of $n=1$ in \cite{Sy:mmlim}.
The following lemma plays a crucial role in the proof.
\begin{lem}\label{keylem}
Let $\{X_N\}_{N=1}^\infty, \{Y_N\}_{N=1}^\infty$ be sequences of mm-spaces satisfying the following  conditions:
\begin{enumerate}[{\rm ({C}1)}]
\item\label{ass1}
$\cP_{Y_N}$ converges weakly to a pyramid $\cP_\infty$ as $N \to \infty$.
\item\label{ass2}
For any $N, l \in \N$ with $N \geq l$, there exists a $1$-Lipschitz map 
$p^N_l: X_N \to Y_l$ such that 
$\lim_{N \to \infty} \dP((p^N_l)_\# \mu_{X_N},\mu_{Y_l})=0$.
\item\label{ass3}
For any $N \in \N$, 
there exists a subset $Y_N' \subset Y_N$ such that 
\begin{enumerate}[$($a$)$]
\item\label{ass31}
$\lim_{N \to \infty}\mu_{Y_N}(Y_N')=1$,
\item\label{ass32}
there exists a Lipschitz map 
$\Phi^N: Y_N' \to X_N$ 
pushing $\mu_{Y'_N}:=(\mu_{Y_N}(Y'_N))^{-1}\mu_{Y_N}|_{Y'_N}$
forward to $\mu_{X_N}$ such that
its smallest Lipschitz constant tends to $1$ as $N \to \infty$.
\end{enumerate}
\end{enumerate}
Then $\cP_{X_N}$ converges weakly to $\cP_\infty$ as $N \to \infty$.
\end{lem}
\begin{proof}
Theorem \ref{thm:emb}\eqref{cpt} ensures the existence of a subsequence $\{X_{N_m}\}_{m=1}^\infty \subset \{X_N\}_{N=1}^\infty$ such that 
$\cP_{X_{N_m}}$ converges weakly to a pyramid $\cP$ as $m \to \infty$.
It suffices to prove $\cP=\cP_{\infty}$ for any such $\cP$.
We deduce (C\ref{ass2}) from $\cP_\infty \subset \cP$.

To show the converse including relation $\cP\subset \cP_\infty$,
let us regard $Y_N'=(Y_N', d_{Y_N}, \mu_{Y_N'})$ as an mm-space. 
By (C\ref{ass31}), we compute
\[
\lim_{N \to \infty} \dTV( \mu_{Y'_N},  \mu_{Y_N} )
=\lim_{N \to \infty}\frac{1}{2} \int_{Y_N}\left| \frac{\mathbf{1}_{Y'_N}}{\mu_{Y_N}(Y_N')}-1\right| d\mu_{Y_N}
=\lim_{N \to \infty}(1-\mu_{Y_N}(Y_N'))=0,
\]
where $\mathbf{1}_{Y'_N}$ stands for the  indicator function of $Y_N'$. 
This with Corollary \ref{cor:rho-dTV} implies that 
the pyramid associated with $Y_N'$ converges weakly to $\cP_\infty$.
Combining this with (C\ref{ass32}) and Lemma \ref{lem:obvious}(1)
leads to $\cP\subset\cP_\infty$.
\end{proof}
\begin{cor}\label{keycor}
Let $\{X_N\}_{N=1}^\infty, \{Y_N\}_{N=1}^\infty$ be two sequences of mm-spaces 
as in Lemma \ref{keylem}.
Let $G$ be a group acting  isometrically on $X_N, Y_N$ satisfying the following conditions.
\begin{enumerate}[{\rm ({C'}1)}]
\item\label{asss1}
The quotient space $\bar{Y}_N$ of $Y_N$ by the $G$-action
is monotone increasing in $N$ with respect to the Lipschitz order.
\item\label{asss2}
$p^N_l, \Phi^N$ are both $G$-equivariant.
\item\label{asss3}
$Y'_N$ is $G$-invariant {\rm(}i.e., $G \cdot Y'_N = Y'_N${\rm)}.
\end{enumerate}
Then, the pyramids associated with $\bar{X}_N$, $\bar{Y}_N$ both converge weakly
to a common pyramid as $N\to\infty$.
\end{cor}
\begin{proof}
Due to Lemma \ref{lem:obvious}(2) with (C'\ref{asss1}), 
$\{\bar{Y}_N\}_N$ satisfies (C\ref{ass1}).
Combining (C'\ref{asss2}) with Lemma \ref{lem:dP-quotient} yields that 
$(\bar{p}^N_l)_\# \bar{\mu}_{X_N}$ converges weakly to  $\bar{\mu}_{Y_l}$.
Lemma \ref{lem:equivdomi} with (C'\ref{asss2})  
ensures that 
the Lipschitz constant of 
$\bar{p}^N_l$ is $1$ and that of $\bar{\Phi}^N$ tends to $1$.
Combining (C'\ref{asss2}) with (C'\ref{asss3}) leads to $\bar{\Phi}^N_\# \bar{\mu}_{Y'_N}=\bar{\mu}_{X_N}$.
By (C'\ref{asss3}), we have $\bar{\mu}_{Y_N}(\bar{Y'}_N)=\mu_{Y_N}(Y'_N)$.
Thus $\{\bar{X}_N\}_N,\{\bar{Y}_N\}_N$
satisfies all the conditions (C\ref{ass1}--\ref{ass3}),
which completes the proof of the corollary.
\end{proof}
Subsection \ref{ssc:lip} is devoted to construct a Lipschitz map 
$\Phi^{N,n,F}_{\varepsilon,\theta}:X^F_{N,n,\varepsilon,\theta} \to X^F_{N,n}$ with $\theta$ defined in \eqref{theta}, 
for which the smallest Lipschitz constant tends to $1$ as  $N \to \infty$,
with the help of the polar decomposition.
In Subsection \ref{ssc:push},  we demonstrate that the normalized push-forward measure
of $\gamma^{N^F  n}|_{X^F_{N,n}, \varepsilon,\theta}$
by $\Phi^{N,n,F}_{\varepsilon,\theta}$ coincides with $\mu^{N,n,F}$.
We finally apply Lemma \ref{keylem} and Corollary \ref{keycor}
to prove Theorems \ref{thm:Stiefel} and \ref{thm:Gr-pS}
for 
\[
X_N=X_{N,n_N}^F, \quad
Y_N=\Gamma^{N^F n}, \quad
Y_N'=X_{N,n_N,\varepsilon_N,\theta}^F, \quad
p^N_l=\pi^N_l(n),\quad
\Phi^N=\Phi^{N,n,F}_{\varepsilon,\theta}.
\]
\subsection{Lipschitz map from $X_{N,n,\varepsilon, \theta}^F$ to  $X_{N,n}^F$}\label{ssc:lip}
\subsubsection{Nearest point projection}
Let us first 
describe the relation between singular values and polar decompositions
as well as in the case of $F=\C$ proved by Li \cite{Li}.
In the case of $F=\H$, 
we should take account of the fact that $\tr(ZW) \neq\tr(WZ)$ may happen.
For $N \geq n$, set
\[
I^N_n:=
\begin{pmatrix}
I_n \\
0_{N-n,n}
\end{pmatrix}
\in M_{N,n}^F.
\]
\begin{lem}[cf.~\cite{Li}*{Lemma 1, Theorems 2,\ 2A}] \label{lemma:Li}
For any $Z_1, Z_2 \in M_{N,n}^F$, 
let 
\[
Z_1=Q_1 H_1=U_1 \Lambda_1 V_1^\ast,\quad
Z_2=Q_2 H_2=U_2 \Lambda_2 V_2^\ast
\]
be their polar and singular value decompositions.  
We then have
\[
\|Z_1-Z_2\| \geq \lambda \|Q_1-Q_2\|, \quad \lambda:=\min\{\lambda_n(Z_1), \lambda_n(Z_2)\} .
\]
\end{lem}
\begin{proof}
For $U:=U_2^\ast U_1\in U^F(N), V:=V_2^\ast V_1  \in U^F(n)$, we see that
\begin{align*}
\|Z_1-Z_2\|
=
\|U_1 \Lambda_1 V_1^\ast -U_2 \Lambda_2 V_2^\ast \|=
\|U \Lambda_1  - \Lambda_2  V \|,  \quad
\|Q_1-Q_2\|
=\|U I^N_n - I^N_nV \|.
\end{align*}
So it suffices to prove $\|U \Lambda_1  - \Lambda_2  V \| \geq \lambda \|U I^N_n - I^N_nV \|$.
Setting
\begin{gather*}
\Sigma_i:=(\Lambda_i,0_{N,N-n}), \ 
%
%
I:=(I^N_n,0_{N,N-n}),\ 
A:= \Sigma_1- \lambda I, \ 
B:= \Sigma_2 -\lambda I \in M_{N}^F,\\
V_n:=\begin{pmatrix}
       V&0_{n,N-n}\\
       0_{N-n,n}&I_{N-n}
     \end{pmatrix} \in U^F(N),
\quad
(\sharp):=2\Re\lr{UI- I V_n }{UA-B V_n}.
\end{gather*}
we have $\|U I^N_n-I^N_n V\|=\|U I-I V_n\|$ and
\[
\|U \Lambda_1 -\Lambda_2 V \|^2
=\|U \Sigma_1 -\Sigma_2 V_n \|^2
=\lambda^2\|U I- IV_n \|^2+ \|UA-B V_n\|^2 +  \lambda \times (\sharp).
\]
Thus it is enough to prove the non-negativity of $(\sharp)$, 
which is expressed as 
\begin{align*}
(\sharp)
&=\tr\left[ (I U^\ast- V_n^\ast I) UA+AU^\ast  (U I- I V_n ) \right] \\
&\quad+\tr \left[ (-IU^\ast+ V_n^\ast I  ) B V_n + V_n^\ast B  (-U I+ I V_n )\right].
\end{align*}
Using the assumptions that  $U, V_n$ are unitary and $A,B, I$ are real diagonal, we compute  
\begin{align*}
\tr (I A)=\tr (IU^\ast UA)=\tr(AU^\ast  U I), \quad
\tr(IB)=\tr ( V_n^\ast I   B V_n )=\tr ( V_n^\ast BI V_n ).
\end{align*}
If we set $X=(x^m_l)_{1\leq m,l \leq N}:= V_n^\ast I U, Y=(y^m_l)_{1\leq m,l \leq N}:=\T U^\ast I\T V_n$, 
then we observe that 
\begin{gather*}
(\tr (AU^\ast  I V_n))^\ast
=\tr (V_n^\ast I UA)
=\tr(XA),\quad
(\tr(  V_n^\ast B  U I))^\ast
=\tr(IU^\ast B V_n)
=\tr(YB),
\end{gather*}
which leads to 
\begin{align*}
(\sharp)
&=\tr\left[(2I- (X+X^\ast))A \right]
+\tr \left[(2I-(Y+Y^\ast))B \right] \\
&=\sum_{l=1}^n(2-(x^l_l+(x^l_l)^\ast))(\lambda_l(Z_1)-\lambda)+\sum_{l=1}^n(2-(y^l_l+(y^l_l)^\ast))(\lambda_l(Z_2)-\lambda).
\end{align*}
We moreover find $\lambda_1(X) \leq 1$, which together with Remark \ref{spectrum}(2) implies  
\begin{gather*}
\lambda_{1}(X+X^\ast)
\leq \lambda_{1}(X)+\lambda_1(X^\ast)
 \leq 2.
\end{gather*}
Since $X+X^\ast$ is Hermitian, 
there exist $P=(p^m_l)_{1 \leq m,l \leq N}\in U^F(N)$ and $\{\xi_m\}_{m=1}^N, \R$ 
with $|\xi_m|=\lambda_{m}(X+X^\ast)$ such that 
$X+X^\ast =P^\ast \diag(\xi_1,\ldots,\xi_N ) P$, 
which implies 
\[
2\Re(x^l_l)=x^l_l+(x^l_l)^\ast=\sum_{m=1}^N (p^m_l )^\ast \xi_m p^m_l  \leq 2\sum_{m=1}^N (p^m_l )^\ast p^m_l=2. 
\]
The same argument proves $\Re(y^m_m) \leq 1$.
This together with $\lambda_l(Z_1),\lambda_m(Z_2) \geq \lambda$ implies
 \[
(\sharp)
= 2\sum_{l=1}^n(1-\Re(x^l_l))(\lambda_l(Z_1)-\lambda)
+ 2\sum_{l=1}^n(1-\Re(y^l_l))(\lambda_l(Z_2)-\lambda)
\geq 0.
\]
This completes the proof.
\end{proof}
We next show that the scaled polar decomposition is the nearest point projection from $M_{N,n}^F$ to $X_{N,n}^F$ 
even if the decomposition is not unique.
\begin{lem}\label{lem:nearest}
Let $Z=QH$ be a polar decomposition of $Z \in M_{N,n}^F$.
For any $r>0$ we have
\[
\min_{Q' \in V_{N,n}^F}\|Z-r Q'\|=\|Z- r Q\|=\sqrt{\sum_{l=1}^n (\lambda_l(Z)-r)^2}.
\]
\end{lem}
\begin{proof}
Let $Z=QH=U \Lambda V^\ast$ be polar and singular value decompositions of $Z$.
By Theorem \ref{thm:decomp}, we have 
\begin{align*}
\|Z-r Q\|^2
=\|U \Lambda V^\ast-r Q\|=\|\Lambda-r U^\ast Q V\|^2 
=\left\|\Lambda-r  I^N_n \right\|
=\sum_{l=1}^n (\lambda_l(Z)-r)^2,
\end{align*}
implying the last equality in the lemma.
For any $Q' \in V_{N,n}^F$, 
we define $X=(x^l_m)_{1 \leq m,l \leq n}\in M_{n}^F$ and $Y \in M_{N-n,n}^F$ as 
\[
\begin{pmatrix}
X\\
Y\\
\end{pmatrix}
:=r U^\ast Q' V.
\]
We then have $X^\ast X+Y^\ast Y=r^2I_n$, hence $\Re (x^l_l) \leq \|x^l_l\| \leq r$,  and 
\[
\|Z-rQ'\|
=\|\Lambda-rU^\ast Q' V\|
=
\left\|
\begin{pmatrix}
\Lambda'-X\\\
-Y\\
\end{pmatrix}
\right\|, \quad
\Lambda':=\diag(\lambda_1(Z),\ldots,\lambda_n(Z)).
\]
A direct computation proves  
\begin{align*}
\left\|
\begin{pmatrix}
\Lambda'-X\\\
-Y\\
\end{pmatrix}
\right\|^2 
&=\|(\Lambda'-X)\|^2+\|Y\|^2\\
&=
\|(\Lambda'-r I_n)-(X-r I_n)\|^2+\tr\{ r^2 I_n-X^\ast X\}  \\
&=\| \Lambda'-r I_n\|^2 
+ \tr\{
\Lambda'(r I_n-X)+(r I_n-X^\ast)\Lambda'
\}\\
&=\sum_{l=1}^n (\lambda_l(Z)-r)^2+2 \sum_{l=1}^n \lambda_l(Z)(r-\Re (x^l_l))
\geq \sum_{l=1}^n (\lambda_l(Z)-r)^2,
\end{align*}
which completes the proof.
\end{proof}
In the same way as for $X_{N,n}^F$,
we define an $(\varepsilon, \theta)$-approximation space $U^F(n)_{\varepsilon,\theta}$ 
of $U^F(n)=V_{n,n}^{F}$ by       
\begin{align*}
U^F(n)_{\varepsilon,\theta}
:=\left\{(z_1,\ldots,z_n) \in M_{n}^F \bigg| \big|\|z_l\|-1\big|<\varepsilon,
                            \left\|\lr{\frac{z_l}{\|z_l\|}}{\frac{z_m}{\|z_m\|}}\right\|< \theta(\varepsilon), \ 1 \leq  l<m \leq n\right\}.
\end{align*}
\begin{cor}\label{prop:l}
For any $n \in \N, \varepsilon >0$ and a function $\theta:(0,\infty)\to(0,\infty)$, 
we set 
\[
L(n,\varepsilon,\theta)
:=\sup_{W \in U^F(n)_{\varepsilon,\theta}}\min_{Q' \in U^F(n)}\|W-Q'\|.
\]
Then, for any $N \geq n$ and $Z \in X_{N,n,\varepsilon,\theta}^F$, we have
\[
\lambda_n(Z) \geq \sqrt{N^F-1}\cdot (1-L(n,\varepsilon,\theta)).
\]
\end{cor}
\begin{proof}
For $Z \in  X^F_{N,n,\varepsilon,\theta}$,
let $Z=QH=U\Lambda V^\ast$ be polar and singular decompositions of $Z$.
Set $\Lambda':=\diag(\lambda_1(Z), \ldots, \lambda_n(Z))$.
It turns out that 
\[
W':=\frac1{\sqrt{N^F-1}}(\Lambda' V^\ast )\in  U^F(n)_{\varepsilon,\theta}
\]
and $W'=V^\ast \cdot (V W')$ is a polar decomposition of $W'$.
We conclude that 
\begin{align*}
\sqrt{N^F-1}-\lambda_n(Z)
&\leq \|Z-\sqrt{N^F-1}\cdot Q\| \\
&=\left\|U \cdot\begin{pmatrix} \Lambda' V^\ast\\ 0_{N-n,n} \end{pmatrix}-\sqrt{N^F-1}\cdot U \cdot\begin{pmatrix} V^\ast\\ 0_{N-n,n} \end{pmatrix} \right\| 
=\sqrt{N^F-1} \cdot \|W'- V^\ast \|\\
&
\leq 
\sqrt{N^F-1} \sup_{W \in U^F(n)_{\varepsilon,\theta}}\min_{Q' \in U^F(n)}\|W-Q'\|
=\sqrt{N^F-1}\cdot L(n,\varepsilon,\theta),
\end{align*}
where we apply Lemma \ref{lem:nearest} for $r =\sqrt{N^F-1}$ in the first inequality and 
for $r=1$ in the last inequality.  This completes the proof.
\end{proof}
Corollary \ref{prop:l} implies that 
if  $L(n,\varepsilon,\theta)<1$,
then the map 
\[
Q^F : X^F_{N,n,\varepsilon,\theta} \ni Z \mapsto Q^F(Z) \in V_{N,n}^F
\]
is well-defined, where $Z=Q^F(Z)H$ is the polar decomposition of $Z$.
\begin{lem}\label{lem:lip}
Assume $L(n,\varepsilon,\theta)<1$.
Then, the map
\[
\Phi^{N,n,F}_{\varepsilon,\theta}: X^F_{N,n,\varepsilon,\theta} \ni Z
\mapsto 
\sqrt{N^F-1}\cdot  Q^F(Z) \in X_{N,n}^F
\]
has Lipschitz constant at most $(1-L(n,\varepsilon,\theta))^{-1}$.
\end{lem}
\begin{proof}
For any $Z, W \in X^F_{N,n,\varepsilon,\theta}$, 
Lemma \ref{lemma:Li} and Corollary \ref{prop:l} together imply
\begin{align*}
\|\Phi^{N,n,F}_{\varepsilon,\theta}(Z)-\Phi^{N,n,F}_{\varepsilon,\theta}(W)\|
&=\sqrt{N^F-1}\cdot\|Q^F(Z)-Q^F(W)\|\\
&\le \frac{\sqrt{N^F-1}}{\min\{\lambda_n(Z), \lambda_n(W)\} }\|Z-W\|
\leq \frac{1}{1-L(n,\varepsilon,\theta)} \|Z-W\|.
\end{align*}
\end{proof}
%
\subsubsection{Condition for $L(n,\varepsilon,\theta)<1$}
Given any $\delta>0$, we define the positive monotone increasing functions on $[0,1)$ by 
\[
\varphi_\delta(s):=\frac{(\delta+s)^2}{1-s}, \quad
R_\delta(s):=\frac{2(\delta+s)}{ 1-s}.
\]
Moreover,  for any $\sigma\in[0,1]$, we set 
\begin{gather*}
 s_l=s_l(\delta):=\begin{cases}
 0 & \text{\ if\ }l=0,\\
 s_{l-1}+\varphi_\delta(s_{l-1}) & \text{\ if\ } 1 \leq l \leq n_\sigma,
 \end{cases}\\
n_\sigma=n_\sigma(\delta):=\max\left\{ n \in \N \cup\{0\}\ |\ s_{n}(\delta) <\sigma \right\}+1.
\end{gather*}
It follows that $s_1(\delta)=\delta^2$ and
$s_{n_\sigma-1} < \sigma \leq s_{n_\sigma}$.
\begin{lem}\label{lem:itere}
For any $\delta>0$ and $\sigma\in[0,1)$, we have
\[
\log_{1+R_\delta(\sigma)}\left(1+R_\delta(\sigma)\frac{\sigma}{\delta^2} \right)
\leq
{n_\sigma}(\delta)
<
\log_{1+R_\delta(0)}\left(1+ R_\delta(0)\frac{\sigma}{\delta^2} \right)+1.
\]
The last inequality also holds true for $\sigma=1$.
\end{lem}
\begin{proof}
For $1 \leq l \leq n_\sigma-1$, 
we observe that  $s_{l}<\sigma$ and 
\begin{gather*}
\varphi_\delta(s_{l})
=\varphi_\delta(s_{l-1}+\varphi_\delta(s_{l-1}))
=\frac{(\delta+s_{l-1}+\varphi_\delta(s_{l-1}))^2}{1-(s_{l-1}+\varphi_\delta(s_{l-1}))}
=\varphi_\delta(s_{l-1}) \left(1+ \frac{2(\delta+s_l)}{ 1-s_{l}} \right),\\
R_\delta(0)
\leq
\frac{2(\delta+s_l)}{ 1-s_{l}}
\leq 
R_\delta(\sigma),
\end{gather*}
and thereby,
\begin{align*}
s_{l+1}
=s_{l}+\varphi_\delta(s_{l}) 
&\leq 
s_{l-1} + \varphi_\delta(s_{l-1}) (1+(1+R_\delta(\sigma))) \\
&\leq 
s_{0} + \varphi_\delta(s_0) \sum_{m=0}^{l} (1+R_\delta(\sigma))^m 
\leq \delta^2 \frac{(1+R_\delta(\sigma))^{l+1}-1}{R_\delta(\sigma)}.
\end{align*}
In the same way, we estimate $s_{l+1}$ from below and conclude 
\begin{align}\label{eq:ratio}
\frac{(1+R_\delta(0))^{l+1}-1}{R_\delta(0)}
\leq
\frac{s_{l+1}}{\delta^2} 
\leq \frac{(1+R_\delta(\sigma))^{l+1}-1}{R_\delta(\sigma)}.
\end{align}
Combining this with $s_{n_\sigma-1}<\sigma \leq s_{n_\sigma}$ yields
\begin{align*}
\frac{(1+R_\delta(0))^{n_\sigma-1}-1}{R_\delta(0)} \leq \frac{s_{n_\sigma-1}}{\delta^2}
<
\frac{\sigma}{\delta^2}  
\leq \frac{s_{n_\sigma}}{\delta^2}
\leq \frac{(1+R_\delta(\sigma))^{n_\sigma}-1}{R_\delta(\sigma)},
\end{align*}
which provides the desired result.
\end{proof}
\begin{lem}\label{lem:above}
If $0\leq \sigma \leq \delta^2$, then  $n_\sigma(\delta)=1$
and for $\delta^2<\sigma \leq 1$ we have
\[
 n_{\sigma}(\delta) < 1+\frac{\sigma}{\delta^2}.
\]
\end{lem}
\begin{proof}
The first claim follows from the definition of $n_\sigma(\delta)$ and the fact $s_1=\delta^2$.
In the case of $\delta^2<\sigma\leq1$, 
it suffices to prove by Lemma \ref{lem:itere} that
\[
\log_{1+R_\delta(0)}\left(1+ R_\delta(0)\frac{\sigma}{\delta^2} \right)
< \frac{\sigma}{\delta^2}
\Longleftrightarrow
\left(1+ R_\delta(0)\frac{\sigma}{\delta^2} \right)^{\delta^2/\sigma R_\delta(0)}
<\left(1+R_\delta(0)\right)^{1/R_\delta(0)}.
\]
This follows from the monotone increasing property of $r \mapsto (1+1/r)^r$ on $(0,\infty)$ and the condition 
\[
\frac{\delta^2}{\sigma R_\delta(0)}<\frac{1}{R_\delta(0)}.
\]
The proof is completed.
\end{proof}
%
We next consider the different expression of $s_l(\delta)$.
For $0 \leq l  \leq n_\sigma(\delta)$ with $\sigma\in [0,1]$, we set 
\[
c_l=c_l(\delta):=
\begin{cases}
0 & \text{\ if\ }l=0,\\
\displaystyle
\frac{1+\delta \sum_{m=0}^{l-1}c_m^2}{\sqrt{1-\delta^2 \sum_{m=0}^{l-1}c_m^2}} & \text{\ if\ } 1\leq l \leq n_1(\delta).
\end{cases}
\]
\begin{lem}\label{lem:rel}
For $0 \leq l  \leq n_\sigma(\delta)$ with $\sigma\in [0,1]$, we have 
\[
\delta^2 \sum_{m=0}^{l}c_m^2=s_l(\delta).
\]
In particular, the denominator of $c_l$ does not vanish and 
$c_l$ is well-defined.
\end{lem}
\begin{proof}
We prove the lemma by induction on $l$.
In the case of $l=0$, the statement is true since $c_0=0=s_0$.
Assume that the statement is true for $l-1$.
We then have 
\begin{align*}
\delta^2 \sum_{m=0}^{l}c_m^2
&=\delta^2 \sum_{m=0}^{l-1}c_m^2 +\delta^2 c_l^2
=\delta^2 \sum_{m=0}^{l-1}c_m^2 +\frac{(\delta +s_{l-1})^2}{1-s_{l-1}}
=s_{l-1}+\varphi_\delta(s_{l-1})
=s_l,
\end{align*}
ensuring the statement  for $l$.  This completes the proof.
\end{proof}
\begin{thm}\label{lem:1}
For any  $\sigma \in [0,1], n \leq n_\sigma(\theta(\varepsilon))$, we have 
\[
L(n,\varepsilon,\theta)^2=\sup_{W \in U^F(n)_{\varepsilon,\theta}}\min_{Q' \in U^F(n)}\|W-Q'\|^2
<
\left(n \varepsilon^2+ \frac{2n(1+\varepsilon)\sigma}{1+ \sqrt{1-\sigma}}\right).
\]
\end{thm}
\begin{proof}
For any $Z=(z_1,\ldots, z_n) \in U^F(n)_{\varepsilon, \theta}$, 
there exist 
$P \in U^F(n)$ and 
$(\zeta_1,\ldots, \zeta_n) \in U^F(n)_{\varepsilon,\theta}$ with $\|\zeta_l\|=1$ 
such that 
\[
P z_l=\|z_l\|\zeta_l,\quad
\zeta_1=e_1, \quad
\zeta^m_l:=\lr{e_m}{\zeta_l}=0 \quad \text{if $m>l$}.
\]
\begin{clm}
For $1 \leq m\leq n_\sigma(\theta(\varepsilon))$, 
if $m < l \leq n_\sigma(\theta(\varepsilon))$,
then $\|\zeta^m_l\|< \theta(\varepsilon) c_m(\theta(\varepsilon))$.
\end{clm}
We prove the claim by induction on $m$.
If $m=1$, 
we then have $c_1(\theta(\varepsilon))=1$ and, for $l \geq 2$,
\[
\|\zeta^1_l\|
=\|\lr{e_1}{\zeta_l}\|
=\left\| \left\langle \frac{z_1}{\|z_1\|}, \frac{z_l}{\|z_l\|} \right\rangle \right\|
< \theta(\varepsilon) =\theta(\varepsilon) c_1(\theta(\varepsilon)).
\]
%
%

Let us  assume that 
the claim holds true for any $1, \ldots m-1$ with $m<n_\sigma(\theta(\varepsilon))$.
We derive from the assumption and Lemma \ref{lem:rel} that
\[
\|\zeta^m_m\|^2
=\|\zeta_m\|^2-\sum_{\alpha=1}^{m-1} \|\zeta^\alpha_m\|^2
>1-\theta(\varepsilon)^2\sum_{\alpha=1}^{m-1} c_\alpha^2=1-s_{m-1} > 1-\sigma \ge 0,
\]
providing $\zeta^m_m \neq 0$.
Since for any $l>m$ we have
\[
\lr{\zeta_m}{\zeta_l}
=\sum_{\alpha=1}^m (\zeta^\alpha_m)^\ast \zeta^\alpha_l 
=\sum_{\alpha=1}^{m-1} (\zeta^\alpha_m)^\ast \zeta^\alpha_l+(\zeta^m_m)^\ast \zeta^m_l, \quad
\|\lr{\zeta_m}{\zeta_l}\|<\theta(\varepsilon)
\]
and 
the inverse of $(\zeta^m_m)^\ast$ is $\zeta^m_m/\|\zeta^m_m\|^2$, 
we deduce from the claim for $1, \ldots, m-1$ that 
\begin{align*}
\|\zeta^m_l\|
&=
\left\|\frac{\zeta^m_m}{\|\zeta^m_m\|^2}\cdot \left(\lr{\zeta_m}{\zeta_l}-\sum_{\alpha=1}^{m-1} (\zeta^\alpha_m)^\ast \zeta^\alpha_l \right)\right\| 
\leq 
\frac{1}{\|\zeta^m_m\|} \left( \theta(\varepsilon)+ \sum_{\alpha=1}^{m-1} \|\zeta^\alpha_m\| \|\zeta^\alpha_l \|  \right) \\
&<
\frac{ \theta(\varepsilon)+ \theta(\varepsilon) \sum_{\alpha=1}^{m-1} c_\alpha \|\zeta^\alpha_m\| }{\sqrt{1-\sum_{\alpha=1}^{m-1} \|\zeta^\alpha_m\|^2}} 
<
\frac{ \theta(\varepsilon)  \left(1 + \theta(\varepsilon) \sum_{\alpha=1}^{m-1} c_\alpha^2  \right)}{\sqrt{1- \theta(\varepsilon)^2 \sum_{\alpha=1}^{m-1} c_\alpha^2}}  
=\theta(\varepsilon) c_m, 
\end{align*}
where the last inequality follows from 
the monotone increasing property of each varieties of
\[
(r_1,\ldots,r_{m-1}) \mapsto \frac{1}{\sqrt{1-\sum_{\alpha=1}^{m-1} r_\alpha^2 }} \left(\theta(\varepsilon) +\theta(\varepsilon) \sum_{\alpha=1}^{m-1}c_\alpha r_\alpha \right).
\]
This completes the proof of the claim.

It thus holds for any $1 \leq l \leq n_\sigma(\theta(\varepsilon))$ that 
\[
\|\zeta^l_l\|
=\sqrt{1-\sum_{m=1}^{l-1} \|\zeta^m_l\|^2}>\sqrt{1-\theta(\varepsilon)^2\sum_{m=1}^{l-1} c_m^2(\theta(\varepsilon))}
=\sqrt{1-s_{l-1}(\theta(\varepsilon))}>0.
\]
Setting 
\[
P':= P^\ast  \cdot  \left(\frac{\zeta^1_1}{\|\zeta^1_1\|}e_1, \ldots, \frac{\zeta^n_n}{\|\zeta^n_n\|}e_n\right) \in U^{F}(n),
\]
we see that 
\begin{align*}
\min_{Q \in U^F(n)} \|Z-Q\|^2
&\leq \|Z-P'\|^2
=\sum_{l=1}^n \left\| \|z_l\| \zeta_l- \frac{\zeta^l_l}{\|\zeta^l_l\|}e_l \right\|^2 
=\sum_{l=1}^n \left( \|z_l\|^2 +1-2\|z_l\| \|\zeta^l_l\| \right) \\
&<
\sum_{l=1}^n \left( (\|z_l\|-1)^2+ 2\|z_l\| -2\|z_l\| \sqrt{1-s_{l-1}(\theta(\varepsilon))} \right) \\
&< \sum_{l=1}^n \left( \varepsilon^2+2(1+\varepsilon)(1-\sqrt{1-\sigma} )\right)
= \left(n \varepsilon^2+ \frac{2n(1+\varepsilon)\sigma}{1+ \sqrt{1-\sigma}}\right).
\end{align*}
By the arbitrariness of $Z \in U^F(n)_{\varepsilon, \theta}$,
the proof of the lemma is complete.
\end{proof}
\begin{rem}\label{rem:orth}
In Theorem \ref{lem:1}, 
we estimate $L(n,\varepsilon,\theta)^2$ from above by the sum of two terms:
the first one is due to the error of the lengths between each column vectors $Z \in U^F(n)_{\varepsilon,\theta}$ and $Q^{F}(Z)$, 
and the second one is due to the error of the angles between each column vectors $Z$ and $Q^{F}(Z)$.
If $n_\sigma(\theta(\varepsilon)) \to \infty$ as $\varepsilon \to 0$, then 
we require $\theta(\varepsilon)^2<\sigma$ by Lemma \ref{lem:above}.
Therefore the second term is larger than the first one if $\varepsilon\leq \theta(\varepsilon)$, 
which holds for the $\theta$ defined in \eqref{theta}.
This is according to the fact that  
the condition $L(n,\varepsilon,\theta)<1$ guarantees that 
the rank of any elements in $U^F(n)_{\varepsilon,\theta}$ equals to $n$, 
and 
the rank of a matrix is stable to the perturbation on the lengths of their column vectors,
but extremely frail against the perturbation on the angles between their column vectors.
Indeed, for sufficiently large $n$, 
there exists $\{x_l\}_{l=1}^{2n} \subset \R^n$ such that 
each angle between any two distinct vectors is close to $\pi/2$ (cf.~\cite{CFG}*{Theorem 6}).
\end{rem}
\subsection{ $\mu^{N,n,F}$ as a push forward measure of $\gamma^{N^F n}$}\label{ssc:push}
\begin{lem}\label{lem:push}
We have 
\[
\mu^{N,n,F}=(\Phi^{N,n,F}_{\varepsilon,\theta})_\# \omega^{N,n,F}_{\varepsilon,\theta}, \quad
\text{where}\quad\omega^{N,n,F}_{\varepsilon,\theta}:=
\frac{\gamma^{N^{F}n}|_{X^F_{N,n,\varepsilon,\theta}}}{\gamma^{N^{F}n}(X^F_{N,n,\varepsilon,\theta})}.
\]
\end{lem}
\begin{proof}
%
%
For any $U \in U^F(n)$, 
it is easy to check that $\mathcal{U}_{n}^U$ commutes with the $\R$-multiplication
and we have already seen the commutativity of $\mathcal{U}_{n}^U$ with $Q^F$ in Remark \ref{spectrum}(1).
This means that $\mathcal{U}_{n}^{U^\ast}$ commutes with $\Phi^{N,n,F}_{\varepsilon,\theta}$.
In addition,  the isometric property of $\mathcal{U}_m^U$ enables us to regard $\mathcal{U}_{m}^U \in U^\R(N^F m)$.
We also have 
$X^F_{N,n,\varepsilon,\theta}=\mathcal{U}_{n}^U(X^F_{N,n,\varepsilon,\theta})$.
These facts with the $U^\R(N^F n)$-invariance of $\gamma^{N^F n}$ together yield that for any Borel subset $B \subset X^F_{N,n,\varepsilon,\theta}$,
\begin{align*}
\gamma^{N^F n}(X^F_{N,n,\varepsilon,\theta})\cdot (\Phi^{N,n,F}_{\varepsilon,\theta})_\# \omega^{N,n,F}_{\varepsilon,\theta} (\mathcal{U}_n^U (B)) 
&=\gamma^{N^F n} (X^F_{N,n,\varepsilon,\theta}\cap (\Phi^{N,n,F}_{\varepsilon,\theta})^{-1} (\mathcal{U}_{n}^U(B)) \\
&= \gamma^{N^F n} (\mathcal{U}_{n}^U(X^F_{N,n,\varepsilon,\theta})\cap (\mathcal{U}_{n}^{U^\ast} \circ \Phi^{N,n,F}_{\varepsilon,\theta})^{-1} (B))\\
&=\gamma^{N^F n} (\mathcal{U}_{n}^U ( X^F_{N,n,\varepsilon,\theta} \cap (\Phi^{N,n,F}_{\varepsilon,\theta})^{-1} (B) ) ) \\
&=\gamma^{N^F n} ( X^F_{N,n,\varepsilon,\theta} \cap (\Phi^{N,n,F}_{\varepsilon,\theta})^{-1} (B) ) \\
&=\gamma^{N^F n} \cdot  (\Phi^{N,n,F}_{\varepsilon,\theta})_\#\omega^{N,n,F}_{\varepsilon,\theta} (B).
\end{align*} 
By Proposition \ref{prop:haar}, 
this completes the proof of the lemma.
\end{proof}
\subsection{Convergence of pyramids of (projective) Stiefel and Grassmann manifolds}
\begin{lem} \label{lem:Lip-const}
Let $\theta$ be a function defined in \eqref{theta}.
If $\lim_{N \to \infty}n_N/(N-1)^3=0$,
then the smallest Lipschitz constant of $\Phi^{N,n_N,F}_{\varepsilon,\theta}$ tends to $1$ as 
$N \to \infty$.
\end{lem}
\begin{proof}
If $L(n_{N}, \varepsilon_N, \theta) <1$, then 
$\Phi^{N,n,F}_{\varepsilon,\theta}$ is well-defined and 
has Lipschitz constant at most $(1-L(n_{N}, \varepsilon_N, \theta))^{-1}$
due to Lemma \ref{lem:lip}.
It thus suffices to prove $\lim_{N\to0}L(n_{N}, \varepsilon_N, \theta) =0$.

We use the same notations in Theorem \ref{thm:fullmeas}
as follows:
\[
p_N:=\log_{(N-1)}n_N, \quad
a_N:=\frac14(1+p_N),\quad
\varepsilon_N:=(N-1)^{-a_N}.
\]
We moreover define 
\[
\theta_N:=\theta(\varepsilon_N), \quad
q_N:=\frac{2}{1+p_N}\left(p_N+\frac13\right), \quad
\sigma_N:=\theta_N^{2-q_N}. 
\]
It then holds that
\begin{align*}
&\varepsilon_N \leq \theta_N \leq 3\varepsilon_N, \quad
q_N \leq 1, \quad
a_N(q_N-1)
=\frac{1}{4}\left(p_N-\frac13\right), \quad
a_N q_N=\frac{1}{2}\left(p_N+\frac13\right), \\
&
\lim_{N \to \infty} n_N \varepsilon_N^2
=\lim_{N \to \infty} (N-1)^{p_N-2a_N} 
=\lim_{N \to \infty} (N-1)^{(p_N-1)/2}=0,\\
&\lim_{N \to \infty} \frac{\sigma_N}{\theta_N}
=\lim_{N \to \infty} \theta_N^{1-q_N} 
\leq \lim_{N \to \infty} 3^{1-q_N} \varepsilon_N^{1-q_N}
\leq \lim_{N \to \infty} 3(N-1)^{a_N(q_N-1)}=0,
\end{align*}
where we use \eqref{eq:lim} in the last inequality.
\begin{clm}\label{clm}
For large enough $N$, we have
 \[
n_{N} \leq \frac{\sigma_{N}}{2\theta_{N}^2} \leq n_{\sigma_{N}}(\theta_{N}).
\]
\end{clm}
The first inequality follows from 
\begin{gather*}
 n_N \cdot \frac{\theta_N^2}{\sigma_N}
= n_N \theta_N^{q_N}
\leq  3^{q_N} n_N \varepsilon_N^{q_N}
\leq  3(N-1)^{p_N-a_Nq_N}
= 3(N-1)^{(3p_N-1)/6},
\end{gather*}
hence $\limsup_{N \to \infty} n_N \cdot \theta_N^2/\sigma_N =0$.
We derive the second inequality from Lemma \ref{lem:itere} and
\begin{align*}
\lim_{N \to \infty} \frac{1}{R_{\sigma_N}(\theta_N)}
&=\lim_{N \to \infty} \frac{\theta_N^2}{\sigma_NR_{\sigma_N}(\theta_N)}=\infty,\\
\lim_{N \to \infty}(1+R_{\sigma_N}(\theta_N))^{1/2R_{\sigma_N}(\theta_N)}
&=e^{1/2}<e^1
=\lim_{N \to \infty}\left(1+R_{\sigma_N}(\theta_N)\frac{\sigma_N}{\theta_N^2} \right)^{{\theta_N^2}/{\sigma_N}R_{\sigma_N}(\theta_N)}.
%
\end{align*}
The claim has been proved.

The claim implies that 
\begin{gather*}
\lim_{N \to \infty} n_N \sigma_N
\leq 
\lim_{N \to \infty}\frac{\sigma_N^2}{2\theta_N^2}=0.
\end{gather*}
%
For large enough $N$, we have, by  Theorem \ref{lem:1},
\[
L(n_{N}, \varepsilon_N, \theta)^2<
n_N \varepsilon_N^2+
\frac{2n_N(1+\varepsilon_N)\sigma_{N}}{1+\sqrt{1-\sigma_{N}}}, 
\]
which proves $\lim_{N \to \infty} L_N=0$ as desired.
This completes the proof.
\end{proof}

\begin{proof}[Proof of Theorems \ref{thm:Stiefel} and \ref{thm:Gr-pS}]
We set 
\[
X_N = X_{N,n_N}^F, \quad
Y_N = \Gamma^{N^F n_N}, \quad
Y_N' = X_{N,n_N,\varepsilon_N,\theta}^F, \quad
p^N_l = \pi^N_l(n),\quad
\Phi^N = \Phi^{N,n,F}_{\varepsilon,\theta}
\]
for the $\theta$ defined in \eqref{theta}
and $G=U^F(1)$ or $U^F(n)$.
We already check that the pyramid $\cP_{Y_N}$ converges weakly to $\cP_{\Gamma^\infty}$ as $N \to \infty$.
It is easy to cheek  the $G$-invariance of $Y'_N$.
Since $p^N_l$ is $1$-Lipschitz continuous and  $G$-equivariant, 
the space $\bar{Y}_N$ is monotone increasing in $N$ with respect to the Lipschitz order.
Proposition \ref{thm:MB} implies that 
$\lim_{N \to \infty}\dP((p^N_l)_\# \nu_{X_N}, \nu_{Y_l})=0$.
We confirm that $\Phi^N_\# \nu_{Y'_N} =\nu_{X_N}$ in Lemma \ref{lem:push} and 
$\Phi^N$ is $G$-equivariant in Remark \ref{spectrum}(2).
Theorem \ref{thm:fullmeas} implies $\lim_{N \to \infty}\nu_{Y_N}(Y_N')=1$. 
The smallest Lipschitz constant of $\Phi^N$ tends to $1$ as $N \to \infty$ due to Lemma \ref{lem:Lip-const}.
We thus apply Lemma \ref{keylem} and Corollary \ref{keycor}
to obtain 
\[
\cP_{X^F_{N,n_N}}\to\cP_{\Gamma^\infty}, \ 
\cP_{X^F_{N,n}/U^F(n)}\to\cP_{\Gamma^{\infty n}/U^F(n)}, \ 
\cP_{U^F(1)\backslash X^F_{N,n_N}}\to\cP_{U^F(1)\backslash\Gamma^\infty}, \quad
\text{as}\ N \to \infty.
\]
This completes the proof of the theorems.
\end{proof}
\begin{rem}
(1) Under the assumption \eqref{eq:condi}, or more weakly $\lim_{N\to\infty} n_N^3/N = 0$,
we always have $\lim_{N \to \infty}n_{\sigma_N}(\theta_N)=\infty$ 
for $\theta$ defined in \eqref{theta}
even if $\lim_{N \to \infty} n_N<\infty$.
This follows from Claim \ref{clm} and 
\[
\lim_{N \to \infty} \frac{\sigma_N}{\theta_N^2}
=\lim_{N \to \infty} \theta_N^{-q_N} 
\geq 3^{-q_N}\lim_{N \to \infty} \varepsilon_N^{-q_N}
\geq \frac13\lim_{N \to \infty} (N-1)^{a_Nq_N}=\infty.
\]
(2)
Assume $\lim_{N\to \infty} n_N =\infty$.
If we construct a pair of 
$\varepsilon_N$ and $\theta:(0,\infty) \to (0,\infty)$ 
satisfying 
\[
\lim_{N \to \infty} \varepsilon_N=0, \quad
\lim_{N \to \infty} \gamma^{N^F n} (X_{N,n,\varepsilon_N, \theta}^F)=1
\]
with the help of Lemma \ref{lem:subset}, 
and 
we prove the convergence of $\cP_{X_{N,n}^F}$
by using Lemma \ref{lem:lip}, Theorem \ref{lem:1}
and Lemma \ref{keylem}, where we set $\Phi^N=\Phi^{N,n,F}_{\varepsilon,\theta}$, 
then the assumption
\[
\lim_{N \to \infty} \frac{n_N^3}{N-1}=0,
\]
which is slightly  weaker than \eqref{eq:condi}, is a necessary condition 
according the following three  conditions (a), (b) and (c).
\begin{enumerate}[(a)]
\item\label{2}
As mentioned in Remark \ref{rem:angle},
if we use Lemma \ref{lem:subset},  then we require 
\begin{gather*}  
\left(B^{F}(T_{N,n_N-1})\right)^{n_N-1} \times A^{(N-n_N+1)^{F}}_{\varepsilon_{N,n_N-1} } \subset A^{N^{F}}_{\varepsilon_N }, 
\quad
\lim_{N\to \infty} T_{N,n_N-1}=\infty, \\
\frac{1}{\theta_N^{2}} 
<
1+\frac{(1-\varepsilon_{N,n_N-1} )^2((N-n_N+1)^F-1)}{(n_N-1)T_{N,n_N-1}^2}.
\end{gather*}
\item \label{4}
To use Theorem \ref{lem:1}, we need to assume $n_N \le n_{\sigma_{n_N}}(\theta_N)$.
\item \label{3}
By Lemma \ref{lem:lip} and Theorem \ref{lem:1}, we need
\[
\lim_{N\to \infty} n_N \sigma_N = 0.
\]
\end{enumerate}
By \eqref{2}, we have
\begin{align*}
(n_N-1)T_{N,n_N-1}^2
+(1+\varepsilon_{N,n_N-1})^2((N-n_N+1)^F-1)
<(1+\varepsilon_{N})^2(N^F-1),
\end{align*}
providing
\begin{align*}
\infty=\lim_{N \to \infty} T_{N,n_N-1}^2
=
\lim_{N \to \infty}
\frac{N-1}{n_N}.
\end{align*}
On one hand, if $\lim_{N\to \infty} {\sigma_N}/{\theta_N}>0$ holds, 
then we require, by \eqref{3},
\[
0=\lim_{N \to \infty} n_N \sigma_N=\lim_{N\to \infty} n_N \theta_N \cdot \frac{\sigma_N}{\theta_N},
\]
hence $\lim_{N \to \infty} n_N\theta_N =0$.
On the other hand, 
if we assume
$\lim_{N\to \infty} {\sigma_N}/{\theta_N}=0$,
then we observe from \eqref{4} and Lemma \ref{lem:above} with $\lim_{N \to \infty}n_N=\infty$ 
that $\theta_N^2<\sigma_N$ and 
\[
\lim_{N \to \infty} n_N \theta_N \cdot \frac{\theta_N}{\sigma_N}
\leq 
\lim_{N \to \infty} \left(1+\frac{\sigma_N}{\theta_N^2}\right) \cdot \frac{\theta_N^2}{\sigma_N}
=\lim_{N \to \infty} \frac{\theta_N^2}{\sigma_N} +1 
\leq 2, 
\]
providing   $\lim_{N \to \infty} n_N\theta_N =0$.
We therefore obtain 
\begin{align*}
\lim_{N \to \infty} \frac{n_N^3}{N-1}
=
\lim_{N \to \infty} 
n_N^2 \theta^2_N 
\cdot \frac{n_N}{(N-1)\theta^2_N}
=0.
\end{align*}
\end{rem}

\section{Asymptotic estimate of observable diameter}

For the proof of Corollary \ref{cor:ObsDiam} we need the following
\begin{defn}[Observable diameter of pyramid; \cite{OzSy:pyramid}] \label{defn:ObsDiam-pyramid}
  Let $\kappa > 0$.
  The \emph{observable diameter of a pyramid $\cP$} is defined to be
  \[
  \ObsDiam(\cP;-\kappa)
  := \lim_{\varepsilon\to 0+} \sup_{X \in \cP} \ObsDiam(X;-(\kappa+\varepsilon))
  \quad (\le +\infty).
  \]
\end{defn}

For any mm-space $X$ we have
\[
\ObsDiam(\cP_X;-\kappa) = \ObsDiam(X;-\kappa)
  \]
for any $\kappa > 0$ (see \cite{OzSy:pyramid}).

\begin{thm}[Limit formulas; \cite{OzSy:pyramid}]
  \label{thm:lim}
  Let $\cP$ and $\cP_n$, $n=1,2,\dots$, be pyramids.
  If $\cP_n$ converges weakly to $\cP$ as $n\to\infty$, then, for any $\kappa > 0$,
  \begin{align*}
    \ObsDiam(\cP;-\kappa)
    &= \lim_{\varepsilon\to 0+} \liminf_{n\to\infty}
    \ObsDiam(\cP_n;-(\kappa+\varepsilon)) \\
    &= \lim_{\varepsilon\to 0+} \limsup_{n\to\infty}
    \ObsDiam(\cP_n;-(\kappa+\varepsilon)).
  \end{align*}
\end{thm}

Theorem \ref{thm:lim} together with \cite{Sy:mmg}*{Theorem 2.21}
(see also Corollary \cite{Sy:mmlim}*{\S 5}) leads us to the following

\begin{thm}
We have
\[
\ObsDiam(\cP_{\Gamma^\infty};-\kappa)
= 2D^{-1}\left(1-\frac{\kappa}{2}\right)
\]
for any $\kappa$ with $0 < \kappa < 1$.
\end{thm}

We are now in a position to prove the

\begin{proof}[Proof of Corollary \ref{cor:ObsDiam}]
  (1) follows from Theorems \ref{thm:Stiefel} and \ref{thm:lim}.
  
  Since $G^F_{N,n_N}, \textrm{P}V^F_{N,n_N} \prec V^F_{N,n_N}$,
  we have (2) and (4) from (1).

  We prove (3).
  Let $f : M^F_{1,n}/U^F(n) \to [\,0,+\infty\,)$ be the function defined by
  $f(\bar{z}) := \|z\|$, $\bar{z} \in M^F_{1,n}/U^F(n)$.
  Note that $f$ is an isometry.
  By using polar coordinates, we have
  \[
  d(f_\# \bar{\gamma}^{n^F})(r) = g_{n^F}(r) \; dr,
  \]
  where $g_m$ is the function defined in \S\ref{ssec:Gaussian}.
  Since $g_m(r) \le g_2(1) = e^{-1/2}$, we have
  \begin{align*}
  \ObsDiam((M^F_{N,n}/U^F(n),\bar{\gamma}^{N n^F});-\kappa)
  &\ge \ObsDiam((M^F_{1,n}/U^F(n),\bar{\gamma}^{n^F});-\kappa) \\
  &= \diam(f_\# \bar{\gamma}^{n^F}; 1-\kappa)
  \ge e^{1/2} (1-\kappa),
  \end{align*}
  which together with Theorems \ref{thm:Gr-pS}(1) and \ref{thm:lim} implies (3).

  We prove (5).
  Since $\sqrt{N^F-1} \,\textrm{P}V^F_{N,m_N} = U^F(1)\backslash X^F_{N,m_N} \succ U^F(1)\backslash X^F_{N,1}$,
  it suffices to estimate the observable diameter of $U^F(1)\backslash X^F_{N,1}$ from below.
  By Theorems \ref{thm:Gr-pS}(2) and \ref{thm:lim}, 
  \begin{align*}
  \lim_{N\to\infty} \ObsDiam(U^F(1)\backslash X^F_{N,1};-\kappa)
  &= \ObsDiam(\cP_{U^F(1)\backslash \Gamma^\infty};-\kappa) \\
  &\ge \ObsDiam((U^F(1)\backslash F,\bar\gamma^{1^F});-\kappa),
  \end{align*}
  where the last inequality follows from 
  $\cP_{U^F(1)\backslash \Gamma^\infty} \ni (U^F(1)\backslash F,\bar\gamma^{1^F})$, $1^F = \dim_\R F$.
  We see that $(U^F(1)\backslash F,\bar{\gamma}^{1^F})$ is mm-isomorphic to $([\,0,+\infty\,), g_{1^F}(r)\,dr)$ and therefore
  \[
  \ObsDiam((U^F(1)\backslash F,\bar\gamma^{1^F});-\kappa) = \diam(g_{1^F}(r)\,dr;1-\kappa) \ge e^{1/2}(1-\kappa).
  \]
  This completes the proof.
\end{proof}

Corollary \ref{cor:ObsDiam}(2)(3) together with Theorem \cite{OzSy:pyramid}*{Theorem 1.2}
implies the following

\begin{cor}
Let $n$ be any fixed positive integer.
\begin{enumerate}
\item We have $t_N/\sqrt{N} \to 0$ as $N \to \infty$ if and only if $\{t_N G^F_{N,n}\}$ is a L\'evy family,
i.e., converges weakly to one-point mm-space.
\item We have $t_N/\sqrt{N} \to 0$ as $N \to \infty$ if and only if $\{t_N G^F_{N,n}\}$ infinitely dissipates,
i.e., the associated pyramid converges weakly to $\cX$.
\end{enumerate}
\end{cor}

Note that the same property for many other manifolds was already obtained in \cite{OzSy:pyramid}.

\renewcommand{\thesection}{\Alph{section}}
\setcounter{section}{0}
\section{Appendix: $U^F(N)$ as a subgroup of $U^\R(N^F)$}
For  $z:=z_0+z_1\i+z_2\j+z_3\k \in F^N$, we set 
\[
\Ree(z):=z_0,\quad
\I(z):=z_1,\quad
\J(z):=z_2,\quad
\K(z):=z_3.
\]
It follows that for
$z:=z_0+z_1\i+z_2\j+z_3\k , w:=w_0+w_1\i+w_2\j+w_3\k  \in F^N$,
\[
\lr{
\begin{pmatrix}
      z_0\\
      z_1 \\
      z_2 \\
      z_3
    \end{pmatrix}}
{\begin{pmatrix}
      w_0\\
      w_1 \\
      w_2 \\
      w_3
    \end{pmatrix}}
=\sum_{l=0}^3 z_lw_l
=\Re\lr{z}{w}.
\]
\begin{lem}\label{lem:unit}
Define a map
$\mathcal{O}^F: U^F(N) \hookrightarrow \mathrm{M}_{N^F}(\R)$ 
by 
\begin{gather*}
\mathcal{O}^{F}(U)
\begin{pmatrix}
      z_0\\
      z_1 \\
      z_2 \\
      z_3
\end{pmatrix}
:=
\begin{pmatrix}
\Ree(Uz)\\
\I(Uz)\\
\J(Uz)\\
\K(Uz))
    \end{pmatrix},
\end{gather*}
for $z_0,z_1,z_2,z_3 \in \R^N$,
where $z:=z_0+z_1\i+z_2\j+z_3\k \in F^N$.
Then we have for $U, V \in U^F(N)$,
\[
\mathcal{O}^{F}(U) \in U^\R(N^F), \quad
\mathcal{O}^{F}(UV)=\mathcal{O}^{F}(U)\mathcal{O}^{F}(V), \quad
\mathcal{O}^{F}(U^\ast)=\mathcal{O}^{F}(U)^\ast.
\]
\end{lem}
\begin{proof}
For any $z:=z_0+z_1\i+z_2\j+z_3\k , w:=w_0+w_1\i+w_2\j+w_3\k  \in F^N$ and  $U,V \in U^F(N)$, 
we compute that  
\begin{align*}
\lr{\mathcal{O}^{F}(U)
\begin{pmatrix}
      z_0\\
      z_1 \\
      z_2 \\
      z_3
\end{pmatrix}}
{
\mathcal{O}^{F}(U)
\begin{pmatrix}
      z_0\\
      z_1 \\
      z_2 \\
      z_3
\end{pmatrix}}
=\Re\lr{Uz}{Uw}=\Re\lr{z}{w}
=
\lr{
\begin{pmatrix}
      z_0\\
      z_1 \\
      z_2 \\
      z_3
\end{pmatrix}}%
{\begin{pmatrix}
      w_0\\
      w_1 \\
      w_2 \\
      w_3
    \end{pmatrix}},\\
\mathcal{O}^{F}(UV)
\begin{pmatrix}
      z_0\\
      z_1 \\
      z_2 \\
      z_3
\end{pmatrix}
=
\begin{pmatrix}
\Re(UVz)\\
\I(UVz)\\
\J(UVz)\\
\K(UVz)
\end{pmatrix}
=
\mathcal{O}^{F}(U)
\begin{pmatrix}
\Re(Vz)\\
\I(Vz)\\
\J(Vz)\\
\K(Vz)   
\end{pmatrix}
=
\mathcal{O}^{F}(U)
\mathcal{O}^{F}(V)
\begin{pmatrix}
      z_0\\
      z_1 \\
      z_2 \\
      z_3
\end{pmatrix},
\end{align*}
implying the first two claims.
We also find  that $\mathcal{O}^{F}(I_N)=I_{N^F}$.
This implies that 
\[
 \mathcal{O}^{F}(U)^\ast \mathcal{O}^{F}(U)=I_N
= \mathcal{O}^{F}(U^\ast U)
= \mathcal{O}^{F}(U^\ast) \mathcal{O}^{F}(U),
\]
hence $\mathcal{O}^{F}(U^\ast)=\mathcal{O}^{F}(U)^\ast$.
This completes the proof.
\end{proof}
%


\end{document}